\documentclass[reqno]{amsart}
\usepackage{amsmath, amsthm, amssymb, amstext}

\usepackage[left=3.5cm,right=3.5cm,top=3cm,bottom=3cm]{geometry}
\usepackage{hyperref,xcolor}
\hypersetup{
 pdfborder={0 0 0},
 colorlinks,
}
\usepackage{enumitem}
\allowdisplaybreaks
\usepackage{todonotes}

\usepackage{graphicx}
\usepackage{tikz}
\usepackage{capt-of}
\newcommand{\xticklength}{0.05}

\newtheorem{theorem}{Theorem}
\newtheorem{remark}[theorem]{Remark}
\newtheorem{lemma}[theorem]{Lemma}
\newtheorem{proposition}[theorem]{Proposition}
\newtheorem{corollary}[theorem]{Corollary}

\DeclareMathOperator*{\divergenz}{div}              %
\DeclareMathOperator*{\Ss}{S}

\newcommand{\N}{\mathbb{N}}

\newcommand{\R}{\mathbb{R}}

\newcommand*\diff{\mathop{}\!\mathrm{d}}

\newcommand{\Lp}[1]{L^{#1}(\Omega)}

\newcommand{\Lprand}[1]{L^{#1}(\partial\Omega)}
\newcommand{\Wp}[1]{W^{1,#1}(\Omega)}

\newcommand{\Wpzero}[1]{W^{1,#1}_0(\Omega)}

\newcommand{\lan}{\langle}
\newcommand{\ran}{\rangle}
\newcommand{\eps}{\varepsilon}
\newcommand{\ph}{\varphi}

\newcommand{\into}{\int_{\Omega}}

\newcommand{\weak}{\rightharpoonup}

\newcommand{\Linf}{L^{\infty}(\Omega)}

\renewcommand{\l}{\left}
\renewcommand{\r}{\right}

\newcommand{\WH}{W^{1, \mathcal{H}}(\Omega)}
\numberwithin{theorem}{section}
\numberwithin{equation}{section}

\newcommand{\nc}{\mathcal{N}}
\newcommand{\e}{\mathbb \varepsilon}

\def\le{\leqslant}

\def\phi{\varphi}

\def\ykh#1{\left(#1\right)}

\title[Parametric superlinear double phase problems with singular term]{Parametric superlinear double phase problems with singular term and critical growth on the boundary}

\author[\'{A}.\,Crespo-Blanco]{\'{A}ngel Crespo-Blanco}
\address[\'{A}.\,Crespo-Blanco]{Technische Universit\"{a}t Berlin, Institut f\"{u}r Mathematik, Stra\ss e des 17.\,Juni 136, 10623 Berlin, Germany}
\email{crespo@math.tu-berlin.de}

\author[N.\,S.\,Papageorgiou]{Nikolaos S.\,Papageorgiou}
\address[N.\,S.\,Papageorgiou]{National Technical University, Department of Mathematics, Zografou Campus, Athens 15780, Greece}
\email{npapg@math.ntua.gr}

\author[P.\,Winkert]{Patrick Winkert}
\address[P.\,Winkert]{Technische Universit\"{a}t Berlin, Institut f\"{u}r Mathematik, Stra\ss e des 17.\,Juni 136, 10623 Berlin, Germany}
\email{winkert@math.tu-berlin.de}



\begin{document}

\begin{abstract}
	In this paper we study quasilinear elliptic equations driven by the double phase operator along with a reaction that has a singular and a parametric superlinear term and with a nonlinear Neumann boundary condition of critical growth. Based on a new equivalent norm for Musielak-Orlicz Sobolev spaces and the Nehari manifold along with the fibering method we prove the existence of at least two weak solutions provided the parameter is sufficiently small. 
\end{abstract}

\maketitle

\section{Introduction}

Given a bounded domain $\Omega\subset \R^N $, $N\geq 2$, with Lipschitz boundary $\partial \Omega$, we study the following singular double phase problem with critical growth on the boundary
\begin{equation}\label{problem}
	\begin{aligned}
		-\divergenz\big(|\nabla u|^{p-2} \nabla u+ \mu(x) |\nabla u|^{q-2} \nabla u\big)+\alpha(x)u^{p-1}&= \zeta(x)u^{-\kappa}+\lambda u^{q_1-1} && \text{in } \Omega, \\
		\big(|\nabla u|^{p-2} \nabla u+ \mu(x) |\nabla u|^{q-2} \nabla u\big) \cdot \nu&= -\beta(x)u^{p_* -1} && \text{on } \partial\Omega,
	\end{aligned}
\end{equation}
where $\lambda>0$ and  $\nu(x)$ is the outer unit normal of $\Omega$ at the point $x \in \partial\Omega$. The operator involved is the so-called double phase operator given by
\begin{align*}
	\divergenz \big(|\nabla u|^{p-2} \nabla u+ \mu(x) |\nabla u|^{q-2} \nabla u\big)\quad \text{for }u\in \WH,
\end{align*}
which is related to the energy functional
\begin{align}\label{integral_minimizer}
	\omega \mapsto \int_\Omega \big(|\nabla  \omega|^p+\mu(x)|\nabla  \omega|^q\big)\diff x.
\end{align}
Functionals of type \eqref{integral_minimizer}  have first been studied by Zhikov \cite{Zhikov-1986} in order to provide models for strongly anisotropic materials. The main characteristic of the functional defined in \eqref{integral_minimizer} is the change of ellipticity on the set where the weight function is zero, that is, on the set $\{x\in \Omega: \mu(x)=0\}$. To be more precise, the energy density of \eqref{integral_minimizer} exhibits ellipticity in the gradient of order $q$ on the points $x$ where $\mu(x)$ is positive and of order $p$ on the points $x$ where $\mu(x)$ vanishes. Further results on regularity of minimizers of \eqref{integral_minimizer} can be found in the papers of Baroni-Colombo-Mingione \cite{Baroni-Colombo-Mingione-2015,Baroni-Colombo-Mingione-2018}, Colombo-Mingione \cite{Colombo-Mingione-2015a,Colombo-Mingione-2015b}, De Filippis-Mingione \cite{DeFilippis-Mingione-2021}, Marcellini \cite{Marcellini-1991,Marcellini-1989b} and Ragusa-Tachikawa \cite{Ragusa-Tachikawa-2020}.

We suppose the  following assumptions:
\begin{enumerate}
	\item[(H):]
	\begin{enumerate}
		\item[(i)]
		$1<p<N$, $p<q<p^*$ and $0 \leq \mu(\cdot)\in \Linf$;
		\item[(ii)]
		$0<\kappa<1$ and $q_1 \in (\max\{q,p_* \},p^*)$, where
		\begin{align}\label{critical_exponents}
			p^*=\frac{Np}{N-p}\quad\text{and}\quad p_*=\frac{(N-1)p}{N-p}
		\end{align}
		are the critical exponents to $p$;
		\item[(iii)]
		$\alpha \in \Linf$ with $\alpha(x) \geq 0$ for a.\,a.\,$x\in\Omega$ and $\alpha\not\equiv 0$;
		\item[(iv)]
		$\beta\in L^\infty(\partial\Omega)$ with $\beta(x)\geq 0$ for a.\,a.\,$x\in\partial\Omega$;
		\item[(v)]
		$\zeta \in \Linf$ and $\zeta(x) > 0$ for a.\,a.\,$x\in\Omega$.
	\end{enumerate}
\end{enumerate}

We call a function $u \in\WH$ a weak solution of problem \eqref{problem} if $\zeta(\cdot)u^{-\kappa} h \in \Lp{1}$, $u(x)> 0$ for a.\,a.\,$x\in\Omega$ and 
\begin{equation}\label{weak_solution}
	\begin{split}
		& \into \big(|\nabla u|^{p-2} \nabla u+ \mu(x) |\nabla u|^{q-2} \nabla u\big) \cdot \nabla h \diff x\\
		&\quad +\into \alpha(x)u^{p-1}h\diff x+\int_{\partial\Omega}\beta(x)u^{p_* -1}h\diff \sigma\\
		&= \into \zeta(x)u^{-\kappa} h \diff x+\lambda \into u^{q_1-1}h\diff x
	\end{split}
\end{equation}
is satisfied for all test functions $h \in \WH$. Based on \textnormal{(H)} it is easy to see that the definition of a weak solution in \eqref{weak_solution} is well-defined. Denoting by  $\Theta_\lambda\colon\WH\to \R$ the energy functional corresponding  to problem \eqref{problem}, the main result in this paper is the following theorem.

\begin{theorem}\label{main_result}
	If hypotheses \textnormal{(H)}  hold, then there exists $\lambda^*>0$ such that for all $\lambda \in (0,\lambda^*]$ problem \eqref{problem} has at least two weak solutions $u_\lambda, v_\lambda \in \WH$ such that $\Theta_\lambda(u_\lambda)<0<\Theta_\lambda(v_\lambda)$.
\end{theorem}

The proof of Theorem \ref{main_result} relies on the properties of the Nehari manifold along with a new equivalent norm in the corresponding Musielak-Orlicz Sobolev space $\WH$. In contrast to other works dealing with double phase problems we were able to weaken the usual condition
\begin{align}\label{exponent-inequality}
	\frac{q}{p}<1+\frac{1}{N}.
\end{align}
Such assumption is standard for Dirichlet double phase problems in order to have the equivalent norm $\|\nabla \cdot\|_{\mathcal{H}}$ in $\Wpzero{\mathcal{H}}$. It also guarantees the density of smooth functions in $\Wp{\mathcal{H}}$. Condition \eqref{exponent-inequality} can be replaced in our paper by $q<p^*$ which is equivalent to $ \frac{Nq}{N+q}<p$. Note that such condition is indeed weaker than \eqref{exponent-inequality}. Furthermore, we can relax the assumptions on the weight function $\mu(\cdot)$. Instead of a Lipschitz condition we only need $\mu(\cdot)$ to be bounded, not necessarily continuous.

Based on the new equivalent norm in $\WH$ we were able to suppose critical growth on the boundary of $\Omega$. 
To the best of our knowledge there is only one paper concerning singular double phase problems with nonlinear boundary condition, namely the paper of Farkas-Fiscella-Winkert \cite{Farkas-Fiscella-Winkert-2021} who studied the problem
\begin{equation}\label{problem2}
	\begin{aligned}
		-\divergenz (A(u)) +u^{p-1}+\mu(x)u^{q-1}& = u^{p^{*}-1}+\lambda \l(u^{\gamma-1}+ g_1(x,u)\r) \quad && \text{in } \Omega,\\
		A(u)\cdot \nu & = u^{p_*-1}+g_2(x,u) &&\text{on } \partial \Omega,
	\end{aligned}
\end{equation}
where
\begin{align*}
	\divergenz (A(u)):=\divergenz\big(F^{p-1}(\nabla u)\nabla F(\nabla u) +\mu(x) F^{q-1}(\nabla u)\nabla F(\nabla u)\big)
\end{align*}
is the Finsler double phase operator with a Minkowski space $(\R^N,F)$. They obtain the existence of one weak solution of \eqref{problem2} by applying variational tools and truncation techniques. The treatment is completely different from ours and only one solution is obtained. Also in the case of nonsingular Neumann double phase problems there are only few works. We refer to El Manouni-Marino-Winkert  \cite{El-Manouni-Marino-Winkert-2021}, Gasi\'nski-Winkert \cite{Gasinski-Winkert-2021}, Papageorgiou-R\u{a}dulescu-Repov\v{s} \cite{Papageorgiou-Radulescu-Repovs-2020} and Papageorgiou-Vetro-Vetro \cite{Papageorgiou-Vetro-Vetro-2020}. In \cite{Gasinski-Winkert-2021} the authors study the problem
\begin{equation}\label{problem3}
	\begin{aligned}
		-\divergenz\left(|\nabla u|^{p-2}\nabla u+\mu(x) |\nabla u|^{q-2}\nabla u\right) & =f(x,u)-|u|^{p-2}u-\mu(x)|u|^{q-2}u && \text{in } \Omega,\\
		\left(|\nabla u|^{p-2}\nabla u+\mu(x) |\nabla u|^{q-2}\nabla u\right) \cdot \nu & = g(x,u) &&\text{on } \partial \Omega.
	\end{aligned}
\end{equation} 
Based on the Nehari manifold method it is shown that problem \eqref{problem3} has at least three nontrivial solutions. We point out that the use of the Nehari manifold in \cite{Gasinski-Winkert-2021} is different from ours. Indeed the idea in the current paper is the splitting of the Nehari manifold into three disjoint parts and the two solutions in Theorem \ref{main_result} turn out to be the global minimizers of $\Theta_{\lambda}$ restricted to two of them provided the parameter is sufficiently small. The third one is the empty set.

In general, the use of the fibering method along with the Nehari manifold is a very powerful tool and was initiated by the works of Dr\'{a}bek-Pohozaev \cite{Drabek-Pohozaev-1997} and Sun-Wu-Long \cite{Sun-Wu-Long-2001}. Afterwards several authors applied this method to different problems of singular and nonsingular type. We refer to the works of Alves-Santos-Silva \cite{Alves-Santos-Silva-2021}, Lei \cite{Lei-2018}, Liu-Dai-Papageorgiou-Winkert \cite{Liu-Dai-Papageorgiou-Winkert-2021}, Liu-Winkert \cite{Liu-Winkert-2021}, Mukherjee-Sreenadh \cite{Mukherjee-Sreenadh-2019}, Papageorgiou-Repov\v{s}-Vetro \cite{Papageorgiou-Repovs-Vetro-2021},  Papageorgiou-Winkert \cite{Papageorgiou-Winkert-2021}, Wang-Zhao-Zhao \cite{Wang-Zhao-Zhao-2013} and Yang-Bai \cite{Yang-Bai-2020}.

For existence results for double phase problems with homogeneous Dirichlet boundary condition we refer to the papers of Colasuonno-Squassina \cite{Colasuonno-Squassina-2016} (eigenvalue problem for the double phase operator), Farkas-Winkert \cite{Farkas-Winkert-2021} (Finsler double phase problems), Gasi\'nski-Papa\-georgiou \cite{Gasinski-Papageorgiou-2019} (locally Lipschitz right-hand side), Gasi\'nski-Winkert \cite{Gasinski-Winkert-2020a,Gasinski-Winkert-2020b} (convection and superlinear problems), Liu-Dai \cite{Liu-Dai-2018} (Nehari manifold approach),
Perera-Squassina \cite{Perera-Squassina-2018} (Morse theoretical approach), Zeng-Bai-Gasi\'nski-Winkert \cite{Zeng-Bai-Gasinski-Winkert-2020, Zeng-Gasinski-Winkert-Bai-2020} (multivalued obstacle problems) and the references therein. Finally, we mention the nice overview article of Mingione-R\u{a}dulescu \cite{Mingione-Radulescu-2021}  about recent developments for problems with nonstandard growth and nonuniform ellipticity.

\section{Preliminaries }

In this section we will present the main properties of Musielak-Orlicz spaces $\Lp{\mathcal{H}}$ and $\WH$, respectively and equip the space $\WH$ with a new equivalent norm. 

The usual Lebesgue spaces $\Lp{r}$ and $L^r(\Omega;\R^N)$ will be endowed with the norm $\|\cdot\|_r$  and the boundary Lebesgue spaces are denoted by $\Lprand{r}$ with norm $\|\cdot\|_{r,\partial\Omega}$ whenever $1\leq r\leq \infty$. The corresponding  Sobolev spaces are denoted by $\Wp{r}$ for $1<r<\infty$ with the equivalent norm  
\begin{align*}
	\|u\|_{1,r}=\l(\|\nabla u\|_{r}^{r}+\int_{\Omega} \alpha(x) |u|^r\diff  x\r)^\frac{1}{r},
\end{align*}
where $\alpha$ fulfills hypothesis \textnormal{(H)(iii)}. The proof for such result is similar to those proofs as in Papageorgiou-Winkert \cite[Proposition 2.8]{Papageorgiou-Winkert-2019}, \cite[Proposition 4.5.34]{Papageorgiou-Winkert-2018}.

Let hypothesis \textnormal{(H)(i)} be satisfied and let $\mathcal{H}\colon \Omega \times [0,\infty)\to [0,\infty)$ be defined by
\begin{align*}
	\mathcal H(x,t)= t^p+\mu(x)t^q.
\end{align*}
Let $M(\Omega)$ be the space of all measurable functions $u\colon\Omega\to\R$. As usual, we identify two such functions which differ on a Lebesgue-null set. Then,  the Musielak-Orlicz Lebesgue space $L^\mathcal{H}(\Omega)$ is given by
\begin{align*}
	L^\mathcal{H}(\Omega)=\left \{u\in M(\Omega)\,:\,\varrho_{\mathcal{H}}(u)<+\infty \right \}
\end{align*}
equipped with the Luxemburg norm
\begin{align*}
	\|u\|_{\mathcal{H}} = \inf \left \{ \tau >0\,:\, \varrho_{\mathcal{H}}\left(\frac{u}{\tau}\right) \leq 1  \right \},
\end{align*}
where the modular function is given by
\begin{align*}
	\varrho_{\mathcal{H}}(u):=\into \mathcal{H}(x,|u|)\diff x=\into \big(|u|^{p}+\mu(x)|u|^q\big)\diff x.
\end{align*}

In addition, we define the seminormed space
\begin{align*}
	L^q_\mu(\Omega)=\left \{u\in M(\Omega)\,:\,\into \mu(x) |u|^q \diff x< +\infty \right \}
\end{align*}
endowed with the seminorm
\begin{align*}
	\|u\|_{q,\mu} = \left(\into \mu(x) |u|^q \diff x \right)^{\frac{1}{q}}.
\end{align*}

The Musielak-Orlicz Sobolev space $W^{1,\mathcal{H}}(\Omega)$ is defined by
\begin{align*}
	W^{1,\mathcal{H}}(\Omega)= \Big \{u \in L^\mathcal{H}(\Omega) \,:\, |\nabla u| \in L^{\mathcal{H}}(\Omega) \Big\}
\end{align*}
equipped with the norm
\begin{align*}
	\|u\|_{1,\mathcal{H}}= \|\nabla u \|_{\mathcal{H}}+\|u\|_{\mathcal{H}},
\end{align*}
where $\|\nabla u\|_\mathcal{H}=\|\,|\nabla u|\,\|_{\mathcal{H}}$. We know that $L^\mathcal{H}(\Omega)$ and $W^{1,\mathcal{H}}(\Omega)$ are uniformly convex and so reflexive Banach spaces, see  Colasuonno-Squassina \cite[Proposition 2.14]{Colasuonno-Squassina-2016} or Harjulehto-H\"{a}st\"{o} \cite[Theorem 6.1.4]{Harjulehto-Hasto-2019}.

The following proposition states the main embeddings for the spaces $\Lp{\mathcal{H}}$ and $\Wp{\mathcal{H}}$, see Gasi\'nski-Winkert \cite[Proposition 2.2]{Gasinski-Winkert-2021} or Crespo-Blanco-Gasi\'nski-Harjulehto-Winkert \cite[Propositions 2.17 and 2.19]{Crespo-Blanco-Gasinski-Harjulehto-Winkert-2021}

\begin{proposition}\label{proposition_embeddings}
	Let \textnormal{(H)(i)} be satisfied and let $p^*$ as well as $p_*$ be the critical exponents to $p$ as given in \eqref{critical_exponents}. Then the following embeddings hold:
	\begin{enumerate}
		\item[\textnormal{(i)}]
		$\Lp{\mathcal{H}} \hookrightarrow \Lp{r}$ and $\WH\hookrightarrow \Wp{r}$ are continuous for all $r\in [1,p]$;
		\item[\textnormal{(ii)}]
		$\WH \hookrightarrow \Lp{r}$ is continuous for all $r \in [1,p^*]$ and compact for all $r \in [1,p^*)$;
		\item[\textnormal{(iii)}]
		$\WH \hookrightarrow \Lprand{r}$ is continuous for all $r \in [1,p_*]$ and compact for all $r \in [1,p_*)$;
		\item[\textnormal{(iv)}]
		$\Lp{\mathcal{H}} \hookrightarrow L^q_\mu(\Omega)$ is continuous;
		\item[\textnormal{(v)}]
		$\Lp{q}\hookrightarrow\Lp{\mathcal{H}} $ is continuous.
	\end{enumerate}
\end{proposition}

Furthermore, we introduce the seminormed space
\begin{align*}
	W^q_\mu(\Omega;\R^N)=\left \{u\in L^q_\mu(\Omega)\,:\,\into \mu(x) |\nabla u|^q \diff x< +\infty \right \}
\end{align*}
endowed with the seminorm
\begin{align*}
	\|\nabla u\|_{q,\mu} = \left(\into \mu(x) |\nabla u|^q \diff x \right)^{\frac{1}{q}},
\end{align*}
see Proposition \ref{proposition_embeddings}\textnormal{(iv)}.

Next we prove the existence of two equivalent norms in $\WH$ which will be useful in our treatment. In the following we use the seminorms 
\begin{align*}
	\|u\|_{r_1,\theta_1}=\left(\into \theta_1(x)|u|^{r_1}\diff x\right)^{\frac{1}{r_1}}\quad\text{and}\quad \|u\|_{r_2,\theta_2,\partial\Omega}=\left(\int_{\partial\Omega} \theta_2(x)|u|^{r_2}\diff \sigma \right)^{\frac{1}{r_2}},
\end{align*}
where we suppose
\begin{enumerate}
	\item[(H'):]
	\begin{enumerate}
		\item[\textnormal{(i)}]
			$1<p<N$, $p<q<p^*$ and $0 \leq \mu(\cdot)\in \Linf$;
		\item[\textnormal{(ii)}]
			$1\leq r_1\leq p^*$ and $1\leq r_2\leq p_*$;
		\item[\textnormal{(iii)}]
			$\theta_1 \in \Linf, \theta_1(x) \geq 0$ for a.\,a.\,$x\in\Omega$;
		\item[\textnormal{(iv)}]
			$\theta_2\in L^\infty(\partial\Omega)$, $\theta_2(x) \geq 0$ for a.\,a.\,$x\in\partial\Omega$;
		\item[\textnormal{(v)}]
			$\theta_1 \not\equiv 0$ or $\theta_2 \not\equiv 0$.
	\end{enumerate}
\end{enumerate}

\begin{proposition}\label{proposition_equivalent_norm}
	If hypotheses \textnormal{(H')} hold, then
		\begin{align*}
			\|u\|_{1,\mathcal{H}}^{\circ}&=\|\nabla u\|_\mathcal{H}+	\|u\|_{r_1,\theta_1}+\|u\|_{r_2,\theta_2,\partial\Omega}\\
			\|u\|_{1,\mathcal{H}}^{*}&=\inf\l\{\tau>0\,:\, \into \l(\l(\frac{|\nabla u|}{\tau}\r)^p+\mu(x)\l(\frac{|\nabla u|}{\tau}\r)^q\r)\diff x+\into \theta_1(x) \l(\frac{|u|}{\tau}\r)^{r_1}\diff x\r.\\
			& \l. \qquad\qquad\qquad\quad +\int_{\partial\Omega}\theta_2(x)\l(\frac{|u|}{\tau}\r)^{r_2}\diff \sigma\leq 1   \r\},
	\end{align*}
	are both equivalent norms on $\WH$.
\end{proposition}

\begin{proof}
	We prove the result in the critical case, that is, $r_1=p^*$ and $r_2=p_*$, the other cases work similarly.
	First, it is straightforward  to show that $\|\cdot\|_{1,\mathcal{H}}^{\circ}$ and $\|\cdot\|_{1,\mathcal{H}}^{*}$ are norms on $\WH$. 
	
	Applying Proposition \ref{proposition_embeddings}\textnormal{(ii)}, \textnormal{(iii)} gives for $u \in \WH$
	\begin{align*}
		\|u\|_{1,\mathcal{H}}^{\circ}
		&\leq \|\nabla u\|_\mathcal{H}+	\|\theta_1\|_\infty^{\frac{1}{p^*}} \|u\|_{p^*}+\|\theta_2\|_\infty^{\frac{1}{p_*}}\|u\|_{p_*,\partial\Omega}\\
		&\leq \|\nabla u\|_\mathcal{H}+	C_\Omega\|\theta_1\|_\infty^{\frac{1}{p^*}} \|u\|_{1,\mathcal{H}}+C_{\partial \Omega}\|\theta_2\|_\infty^{\frac{1}{p_*}}\|u\|_{1,\mathcal{H}}\\
		& \leq C_1 \|u\|_{1,\mathcal{H}},
	\end{align*}
	where $C_\Omega, C_{\partial\Omega}>0$ are the embedding constants and $C_1>0$. 
	
	Let us now show that
	\begin{align}\label{embedding_1}
		\|u\|_\mathcal{H} \leq c \|u\|_{1,\mathcal{H}}^{\circ}
	\end{align}
	for some $c>0$. Arguing indirectly, we suppose that \eqref{embedding_1} is not true. Then there exists a sequence $\{u_n\}_{n\in\N}\subset \WH$ such that
	\begin{align}\label{embedding_2}
		\|u_n\|_\mathcal{H} > n\|u_n\|_{1,\mathcal{H}}^{\circ}\quad\text{for all }n\in\N.
	\end{align}
	We set $y_n=\frac{u_n}{\|u_n\|_{\mathcal{H}}}$ which gives $\|y_n\|_{\mathcal{H}}=1$. From \eqref{embedding_2} we then obtain
	\begin{align}\label{embedding_3}
		\frac{1}{n}>\|y_n\|_{1,\mathcal{H}}^{\circ}.
	\end{align}
	Since $\|\nabla \cdot\|_\mathcal{H}+\|\cdot\|_\mathcal{H}$ is the norm of $\WH$, we see that $\{y_n\}_{n\in\N}\subset \WH$ is bounded. Hence, we may assume that
	\begin{align}\label{embedding_4}
		y_n\weak y \quad\text{in }\WH\quad\text{and}\quad y_n\weak y \quad\text{ in }\Lp{p^*} \text{ and }\Lprand{p_*},
	\end{align}
	see Proposition \ref{proposition_embeddings}\textnormal{(ii)}, \textnormal{(iii)}. Moreover, from \textnormal{(H')(i)} we know that $q<p^*$. So from Proposition \ref{proposition_embeddings}\textnormal{(ii)}, \textnormal{(v)} we have $\Wp{\mathcal{H}} \hookrightarrow \Lp{q}$ compactly and $\Lp{q}\hookrightarrow\Lp{\mathcal{H}} $ continuously. Therefore $y_n\to y$ in $\Lp{\mathcal{H}}$ and since $\|y_n\|_{\mathcal{H}}=1$ it is clear that $y\neq 0$. Passing to the limit in \eqref{embedding_3} as $n\to \infty$ and using \eqref{embedding_4} along with the weak lower semicontinuity of the norm $\|\nabla \cdot\|_{\mathcal{H}}$ and of the seminorms $\|\cdot\|_{p^*,\theta_1}, \|\cdot \|_{p_*,\theta_2,\partial\Omega}$ leads to
	\begin{align}\label{embedding_5}
		0 \geq \|\nabla y\|_\mathcal{H}+ \|y\|_{p^*,\theta_1}+\|y\|_{p_*,\theta_2,\partial\Omega}.
	\end{align}
	From \eqref{embedding_5} we conclude that $y\equiv \tau\neq 0$ is a constant and so we have
	\begin{align*}
		0\geq |\tau|^{\frac{1}{p^*}} \l(\into \theta_1(x)\diff x\r)^\frac{1}{p^*}+|\tau|^{\frac{1}{p_*}} \l(\int_{\partial\Omega} \theta_2(x)\diff \sigma \r)^\frac{1}{p_*}>0
	\end{align*}
	since $\theta_1 \not\equiv 0$ or $\theta_2 \not\equiv 0$ by hypothesis \textnormal{(H')(v)}. This is a contradiction and so \eqref{embedding_1} is true. From this we directly have that
	\begin{align*}
		\|u\|_{1,\mathcal{H}} \leq C_2 \|u\|_{1,\mathcal{H}}^{\circ}
	\end{align*}
	for some $C_2>0$. 
	
	Let us now prove that $\|\cdot\|_{1,\mathcal{H}}^{\circ}$ and $\|\cdot\|_{1,\mathcal{H}}^{*}$ are equivalent. For $u\in\WH$ we have
	\begin{align*}
		& \into \l(\l(\frac{|\nabla u|}{\|u\|_{1,\mathcal{H}}^{\circ}}\r)^p+\mu(x)\l(\frac{|\nabla u|}{\|u\|_{1,\mathcal{H}}^{\circ}}\r)^q\r)\diff x+\into \theta_1(x) \l(\frac{|u|}{\|u\|_{1,\mathcal{H}}^{\circ}}\r)^{p^*}\diff x\\ &\quad +\int_{\partial\Omega}\theta_2(x)\l(\frac{|u|}{\|u\|_{1,\mathcal{H}}^{\circ}}\r)^{p_*}\diff \sigma\\
		& \leq \varrho_{\mathcal{H}}\l(\frac{\nabla u}{\|\nabla u\|_{\mathcal{H}}}\r)+\int_{\Omega}\theta_1(x)\l(\frac{|u|}{\|u\|_{p^*,\theta_1}}\r)^{p^*}\diff x+\int_{\partial\Omega}\theta_2(x)\l(\frac{|u|}{\|u\|_{p_*,\theta_2,\partial\Omega}}\r)^{p_*}\diff \sigma\\
		& =3.
	\end{align*}
	Therefore, $\|u\|_{1,\mathcal{H}}^{*}\leq 3 \|u\|_{1,\mathcal{H}}^{\circ}$. Similarly, we obtain
	\begin{align}\label{embedding_30}
		\begin{split}
		& \into \l(\l(\frac{|\nabla u|}{\|u\|_{1,\mathcal{H}}^{*}}\r)^p+\mu(x)\l(\frac{|\nabla u|}{\|u\|_{1,\mathcal{H}}^{*}}\r)^q\r)\diff x+\into \theta_1(x) \l(\frac{|u|}{\|u\|_{1,\mathcal{H}}^{*}}\r)^{p^*}\diff x\\ &\quad +\int_{\partial\Omega}\theta_2(x)\l(\frac{|u|}{\|u\|_{1,\mathcal{H}}^{*}}\r)^{p_*}\diff \sigma\\
		& \leq \rho_{1,\mathcal{H}}^*\l(\frac{ u}{\|u\|_{1,\mathcal{H}}^{*}}\r),
		\end{split}	
	\end{align}
	where $\rho_{1,\mathcal{H}}^*$ is the corresponding modular to $\|\cdot\|_{1,\mathcal{H}}^{*}$ given by
	\begin{align*}
		\rho_{1,\mathcal{H}}^*(u)=\into \l(|\nabla u|^p+\mu(x)|\nabla u|^q\r)\diff x+\into \theta_1(x)|u|^{p^*}\diff x+\int_{\partial\Omega} \theta_2(x)|u|^{p_*}\diff \sigma.
	\end{align*}
	Note that, for $u \in \Wp{\mathcal{H}}$, the function $\tau \mapsto \rho_{1,\mathcal{H}}^*(\tau u)$ is continuous, convex and even and it is strictly increasing when $\tau \in [0,+\infty)$. So, by definition, we directly obtain
	\begin{align*}
		\|u\|_{1,\mathcal{H}}^{*}=\tau \quad\text{if and only if}\quad \rho_{1,\mathcal{H}}^*\l(\frac{u}{\tau}\r)=1.
	\end{align*}
	 From this and \eqref{embedding_30} we conclude that $\|\nabla u\|_\mathcal{H} \leq \|u\|_{1,\mathcal{H}}^{*}$, $\|u\|_{p^*,\theta_1}\leq \|u\|_{1,\mathcal{H}}^{*}$ and $\|u\|_{r_2,\theta_2,\partial\Omega}\leq \|u\|_{1,\mathcal{H}}^{*}$. Therefore, $\frac{1}{3} \|u\|_{1,\mathcal{H}}^{\circ}\leq \|u\|_{1,\mathcal{H}}^{*}$.
\end{proof}

Assuming hypotheses \textnormal{(H)} we know from Proposition \ref{proposition_equivalent_norm} that
\begin{align*}
	\|u\|&=\inf\l\{\tau>0\,:\, \into \l(\l(\frac{|\nabla u|}{\tau}\r)^p+\mu(x)\l(\frac{|\nabla u|}{\tau}\r)^q\r)\diff x+\into \alpha(x) \l(\frac{|u|}{\tau}\r)^{p}\diff x\r.\\
	& \l. \qquad\qquad\qquad +\int_{\partial\Omega}\beta(x)\l(\frac{|u|}{\tau}\r)^{p_* }\diff \sigma \leq 1  \r\}
\end{align*}
is a norm on $\WH$ which is equivalent to $\|\cdot\|_{1,\mathcal{H}}$ and $\|\cdot\|_{1,\mathcal{H}}^\circ$. The corresponding modular $\rho$ to $\|\cdot\|$ is given by
\begin{align}\label{modular}
	\rho(u)=\into \l(|\nabla u|^p+\mu(x)|\nabla u|^q\r)\diff x+\into \alpha(x)|u|^p\diff x+\int_{\partial\Omega} \beta(x)|u|^{p_* }\diff \sigma
\end{align}
for $u \in \WH$.

The norm $\|\cdot\|$ and the modular function $\rho$ are related as follows. 

\begin{proposition}\label{proposition_modular_properties}
	Let \textnormal{(H)(i)}, \textnormal{(iii)} and \textnormal{(iv)} be satisfied, let $y\in \Wp{\mathcal{H}}$ and let $\rho$ be defined by \eqref{modular}. Then the following hold:
	\begin{enumerate}
		\item[\textnormal{(i)}]
		If $y\neq 0$, then $\|y\|=\lambda$ if and only if $ \rho(\frac{y}{\lambda})=1$;
		\item[\textnormal{(ii)}]
		$\|y\|<1$ (resp.\,$>1$, $=1$) if and only if $ \rho(y)<1$ (resp.\,$>1$, $=1$);
		\item[\textnormal{(iii)}]
		If $\|y\|<1$, then $\|y\|^q\leq \rho(y)\leq\|y\|^p$;
		\item[\textnormal{(iv)}]
		If $\|y\|>1$, then $\|y\|^p\leq \rho(y)\leq\|y\|^q$;
		\item[\textnormal{(v)}]
		$\|y\|\to 0$ if and only if $ \rho(y)\to 0$;
		\item[\textnormal{(vi)}]
		$\|y\|\to +\infty$ if and only if $ \rho(y)\to +\infty$.
	\end{enumerate}
\end{proposition}

The proof of Proposition \ref{proposition_modular_properties} can be done as in Liu-Dai \cite[Proposition 2.1]{Liu-Dai-2018} or Crespo-Blanco-Gasi\'nski-Harjulehto-Winkert \cite[Proposition 2.16]{Crespo-Blanco-Gasinski-Harjulehto-Winkert-2021}.

Let $A\colon \WH\to \WH^*$ be the nonlinear mapping defined by
\begin{align}\label{operator_representation}
	\begin{split}
		\langle A(u),\ph\rangle_{\mathcal{H}} &=\into \big(|\nabla u|^{p-2}\nabla u+\mu(x)|\nabla u|^{q-2}\nabla u \big)\cdot\nabla\ph \diff x\\
		&\quad +\into \alpha(x)|u|^{p-2}u\ph \diff x+\int_{\partial\Omega} \beta(x)|u|^{p_* -2}u\ph \diff \sigma 
	\end{split}
\end{align}
for all $u,\ph\in\WH$ with $\lan\,\cdot\,,\,\cdot\,\ran_{\mathcal{H}}$ being the duality pairing between $\WH$ and its dual space $\WH^*$.  The properties of the operator $A\colon \WH\to \WH^*$ are summarized in the next proposition.

\begin{proposition}\label{proposition_properties_operator_double_phase}
	Let hypotheses \textnormal{(H)(i)}, \textnormal{(iii)} and \textnormal{(iv)} be satisfied. Then, the operator $A$ defined by \eqref{operator_representation} is bounded (that is, it maps bounded sets into bounded sets), continuous, strictly monotone (hence maximal monotone) and it is of type $(\Ss_+)$.
\end{proposition}

The proof of Proposition \ref{proposition_properties_operator_double_phase} is similar to those in Liu-Dai \cite[Proposition 3.1]{Liu-Dai-2018} or Crespo-Blanco-Gasi\'nski-Harjulehto-Winkert \cite[Proposition 3.5]{Crespo-Blanco-Gasinski-Harjulehto-Winkert-2021}. 

\section{Proof of the main result}
This section is concerned with the proof of Theorem \ref{main_result}. For this purpose, we introduce the energy functional $\Theta_\lambda\colon\WH\to\R$ of problem \eqref{problem} given by
\begin{equation*}
	\begin{split}
		\Theta_{\lambda}(u)&=\frac{1}{p} \|u\|_{1,p}^p+\frac{1}{q}\|\nabla u\|_{q,\mu}^q+\frac{1}{p_* }\|u\|_{p_* ,\beta,\partial\Omega}^{p_* }-\frac{1}{1-\kappa}\into \zeta(x)|u|^{1-\kappa}\diff x-\frac{\lambda}{q_1}\|u\|_{q_1}^{q_1}.
	\end{split}
\end{equation*}
It is clear that $\Theta_\lambda$ is not a $C^1$-functional because of the singular term. Next, for $u \in \WH$, we introduce the fibering function $\psi_u\colon[0,+\infty)\to \R$ given by 
\begin{align*}
	\psi_u(t)=\Theta_\lambda (tu)\quad\text{for all }t\geq 0.
\end{align*}
It is easy to see that $\psi_u \in C^\infty((0,\infty))$. Based on this, we can introduce the so-called Nehari manifold related to problem \eqref{problem} which is defined by
\begin{align*}
	\mathcal{N}_\lambda
	&=\l\{u\in\WH\setminus\{0\}\,:\,\| u\|_{1,p}^p+\|\nabla u\|_{q,\mu}^q+\|u\|_{p_* ,\beta,\partial\Omega}^{p_* }=\into \zeta(x)|u|^{1-\kappa}\diff x+\lambda \|u\|_{q_1}^{q_1}\r\}\\
	& =\l\{u\in\WH\setminus\{0\}\,:\, \psi_u'(1)=0\r\}.
\end{align*}
We know that $\mathcal{N}_\lambda$ contains all weak solutions of problem \eqref{problem} but it is smaller than the whole space $\WH$. The advantage of $\mathcal{N}_\lambda$ is the fact that our energy functional $\Theta_{\lambda}$ has nice properties restricted to $\mathcal{N}_\lambda$ which fail globally. Next, we split the manifold $\mathcal{N}_\lambda$ into three disjoint parts in the following way:
\begin{align*}
	\mathcal{N}_\lambda^+
	&=\l\{u \in \mathcal{N}_\lambda: (p+\kappa-1)\| u\|_{1,p}^p+(q+\kappa-1)\|\nabla u\|_{q,\mu}^q+(p_* +\kappa-1)\|u\|_{p_* ,\beta,\partial\Omega}^{p_* }\r.\\
	&\l. \qquad\qquad \qquad -\lambda (q_1+\kappa-1) \|u\|_{q_1}^{q_1 }>0\r\}\\
	& =\l\{u\in\mathcal{N}_\lambda\,:\, \psi_u''(1)>0\r\},\\
	\mathcal{N}_\lambda^\circ
	&=\l\{u \in \mathcal{N}_\lambda: (p+\kappa-1)\| u\|_{1,p}^p+(q+\kappa-1)\|\nabla u\|_{q,\mu}^q+(p_* +\kappa-1)\|u\|_{p_* ,\beta,\partial\Omega}^{p_* }\r.\\
	&\l. \qquad\qquad \qquad =\lambda (q_1+\kappa-1) \|u\|_{q_1}^{q_1 }\r\}\\
	& =\l\{u\in\mathcal{N}_\lambda\,:\, \psi_u''(1)=0\r\},\\
	\mathcal{N}_\lambda^-
	&=\l\{u \in \mathcal{N}_\lambda: (p+\kappa-1)\| u\|_{1,p}^p+(q+\kappa-1)\|\nabla u\|_{q,\mu}^q+(p_* +\kappa-1)\|u\|_{p_* ,\beta,\partial\Omega}^{p_* }\r.\\
	&\l. \qquad\qquad \qquad -\lambda (q_1+\kappa-1) \|u\|_{q_1}^{q_1 }<0\r\}\\
	& =\l\{u\in\mathcal{N}_\lambda\,:\, \psi_u''(1)<0\r\}.
\end{align*}

In general, the energy functional $\Theta_\lambda$ is not coercive on $\WH$, but it is on the manifold $\mathcal{N}_\lambda$ as stated in the next proposition.

\begin{proposition}\label{proposition_coerivity}
	If hypotheses \textnormal{(H)}  hold, then $\Theta_\lambda\big|_{\mathcal{N}_\lambda}$ is coercive.
\end{proposition}

\begin{proof}
	Let $u\in \mathcal{N}_\lambda$ be such that $\|u\|>1$. By the definition of $\mathcal{N}_\lambda$ we have
	\begin{align}\label{prop_1}
		-\frac{\lambda}{q_1}\|u\|_{q_1}^{q_1}=-\frac{1}{q_1}  \| u\|_{1,p}^p-\frac{1}{q_1}\|\nabla u\|_{q,\mu}^q-\frac{1}{q_1}\|u\|_{p_* ,\beta,\partial\Omega}^{p_* } +\frac{1}{q_1}\into \zeta(x)|u|^{1-\kappa}\diff x.
	\end{align}
	Applying \eqref{prop_1}, Proposition \ref{proposition_modular_properties}\textnormal{(iv)} and Theorem 13.17 of Hewitt-Stromberg \cite[p.\,196]{Hewitt-Stromberg-1965} we obtain
	\begin{align}\label{prop_1a}
		\begin{split}
		\Theta_\lambda(u)&=\l[\frac{1}{p}-\frac{1}{q_1} \r]\|u\|_{1,p}^p+\l[\frac{1}{q}-\frac{1}{q_1} \r]\|\nabla u\|_{q,\mu}^q+\l[\frac{1}{p_* }-\frac{1}{q_1} \r]\|u\|_{p_* ,\beta,\partial\Omega}^{p_* }\\
		&\quad +\l[\frac{1}{q_1}-\frac{1}{1-\kappa} \r]\into\zeta(x)|u|^{1-\kappa}\diff x\\
		& \geq c_1\rho(u) +\l[\frac{1}{q_1}-\frac{1}{1-\kappa} \r]\into\zeta(x)|u|^{1-\kappa}\diff x\\
		& \geq c_1 \|u\|^p-c_2\|u\|^{1-\kappa}
		\end{split}
	\end{align}
	since $p<q<q_1$ and $p_* <q_1$ by hypothesis \textnormal{(H)(ii)} and for some positive constants $c_1, c_2$. Thus, the coercivity of $\Theta_\lambda$ on $\mathcal{N}_\lambda$ follows from \eqref{prop_1a}  as $1-\kappa<1<p$.
\end{proof}

Now we are going to prove that the global minimum of $\Theta_\lambda$ on $\mathcal{N}_\lambda^+$ is negative provided $\mathcal{N}_\lambda^+\neq \emptyset$. The nonemptiness of $\mathcal{N}_\lambda^+$ will be proved later in Proposition \ref{proposition_nonemptiness_and_existence}. To this end, let $m_\lambda^+=\inf_{\mathcal{N}_\lambda^+}\Theta_\lambda$.

\begin{proposition}\label{proposition_negative_energy}
	If hypotheses \textnormal{(H)}  hold and if $\mathcal{N}_\lambda^+\neq \emptyset$, then $\Theta_\lambda \big|_{\mathcal{N}_\lambda^+}<0$. In particular,  $m_\lambda^+<0$.
\end{proposition}

\begin{proof}
	Let $u \in \mathcal{N}_\lambda^+\neq\emptyset$. By the definition of $\mathcal{N}_\lambda^+$ we get
	\begin{align}\label{prop_2}
		\lambda \|u\|_{q_1}^{q_1}<\frac{p+\kappa-1}{q_1+\kappa-1}\|u\|_{1,p}^p+\frac{q+\kappa-1}{q_1+\kappa-1}\|\nabla u\|_{q,\mu}^q +\frac{p_* +\kappa-1}{q_1+\kappa-1}\|u\|_{p_* ,\beta,\partial\Omega}^{p_* }.
	\end{align}
	Since $\mathcal{N}_\lambda^+\subset \mathcal{N}_\lambda$ we have by the definition of $\mathcal{N}_\lambda$
	\begin{align}\label{prop_3}
		-\frac{1}{1-\kappa}\into\zeta(x)|u|^{1-\kappa}\diff x =-\frac{1}{1-\kappa} \l(\|u\|_{1,p}^p+\|\nabla u\|_{q,\mu}^q+\|u\|_{p_* ,\beta,\partial\Omega}^{p_* }\r)+\frac{\lambda}{1-\kappa}\|u\|_{q_1}^{q_1}.
	\end{align}
	Applying \eqref{prop_3} and \eqref{prop_2} leads to
	\begin{align*}
		\begin{split}
			\Theta_{\lambda}(u)&=\frac{1}{p} \|u\|_{1,p}^p+\frac{1}{q}\|\nabla u\|_{q,\mu}^q+\frac{1}{p_* }\|u\|_{p_* ,\beta,\partial\Omega}^{p_* }-\frac{1}{1-\kappa}\into\zeta(x)|u|^{1-\kappa}\diff x-\frac{\lambda}{q_1}\|u\|_{q_1}^{q_1}\\
			& =\l[\frac{1}{p} -\frac{1}{1-\kappa}\r]\|u\|_{1,p}^p
			+\l[\frac{1}{q} -\frac{1}{1-\kappa}\r]\|\nabla u\|_{q,\mu}^q
			+\l[\frac{1}{p_* } -\frac{1}{1-\kappa}\r]\|u\|_{p_* ,\beta,\partial\Omega}^{p_* }\\
			&\quad
			+\lambda \l[\frac{1}{1-\kappa}-\frac{1}{q_1}\r]\|u\|_{q_1}^{q_1}\\
			& \leq \l[\frac{-(p+\kappa-1)}{p(1-\kappa)}+\frac{p+\kappa-1}{q_1+\kappa-1}\cdot \frac{q_1+\kappa-1}{q_1(1-\kappa)}\r]\|u\|_{1,p}^p\\
			&\quad +\l[\frac{-(q+\kappa-1)}{q(1-\kappa)}+\frac{q+\kappa-1}{q_1+\kappa-1}\cdot \frac{q_1+\kappa-1}{q_1(1-\kappa)}\r]\|\nabla u\|_{q,\mu}^q\\
			&\quad +\l[\frac{-(p_* +\kappa-1)}{p_* (1-\kappa)}+\frac{p_* +\kappa-1}{q_1+\kappa-1}\cdot \frac{q_1+\kappa-1}{q_1(1-\kappa)}\r]\|u\|_{p_* ,\beta,\partial\Omega}^{p_* }\\
			& = \frac{p+\kappa-1}{1-\kappa}\l[\frac{1}{q_1}-\frac{1}{p}\r]\|u\|_{1,p}^p+\frac{q+\kappa-1}{1-\kappa}\l[\frac{1}{q_1}-\frac{1}{q}\r]\|\nabla u\|_{q,\mu}^q\\
			&\quad +\frac{p_* +\kappa-1}{1-\kappa}\l[\frac{1}{q_1}-\frac{1}{p_* }\r]\| u\|_{p_* ,\beta,\partial\Omega}^{p_* }\\
			&<0, 
		\end{split}
	\end{align*}
	since $p<q<q_1$ and $p_* <q_1$, see hypotheses \textnormal{(H)(i)} and \textnormal{(ii)}. This shows that $\Theta_\lambda \big|_{\mathcal{N}_\lambda^+}<0$ and so $m_\lambda^+<0$.
\end{proof}

The next proposition shows that $\mathcal{N}_\lambda^\circ$ is empty provided the parameter $\lambda>0$ is sufficiently small.

\begin{proposition}\label{proposition_emptiness}
	 If hypotheses \textnormal{(H)}  hold, then there exists $\hat{\lambda}>0$ such that $\mathcal{N}^\circ_\lambda=\emptyset$ for all $\lambda \in (0,\hat{\lambda})$.
\end{proposition}

\begin{proof}
	We argue indirectly and assume that for every $\hat{\lambda}>0$ we can find  $\lambda \in (0,\hat{\lambda})$ such that $\mathcal{N}^\circ_\lambda \neq \emptyset$. This means that for such $\lambda>0$ there exists $u\in \mathcal{N}_\lambda^\circ$ such that
	\begin{align}\label{prop_4}
		(p+\kappa-1)\|u\|_{1,p}^p+(q+\kappa-1)\|\nabla u\|_{q,\mu}^q+(p_* +\kappa-1)\|u\|_{p_* ,\beta,\partial\Omega}^{p_* }=\lambda (q_1+\kappa-1) \|u\|_{q_1}^{q_1}.
	\end{align}
	We know that $\mathcal{N}_\lambda^\circ\subset \mathcal{N}_\lambda$ and so $u \in \mathcal{N}_\lambda$, that is,
	\begin{align}\label{prop_5}
		\begin{split}
		&(q_1+\kappa-1) \|u\|_{1,p}^p+(q_1+\kappa-1) \|\nabla u\|_{q,\mu}^q+(q_1+\kappa-1) \|u\|_{p_* ,\beta,\partial\Omega}^{p_* }\\
		& = (q_1+\kappa-1)\into\zeta(x)|u|^{1-\kappa}\diff x+\lambda (q_1+\kappa-1)\|u\|_{q_1}^{q_1}.
		\end{split}
	\end{align}
	Subtracting \eqref{prop_4} from \eqref{prop_5} we obtain
	\begin{align}\label{prop_6}
		\begin{split}
			&(q_1-p) \|u\|_{1,p}^p+(q_1-q) \|\nabla u\|_{q,\mu}^q+(q_1-p_* ) \|u\|_{p_* ,\beta,\partial\Omega}^{p_* }= (q_1+\kappa-1)\into\zeta(x)|u|^{1-\kappa}\diff x.
		\end{split}
	\end{align}
	Since $0<1-\kappa<1<p<q<q_1$ and $p_* <q_1$ we can use Proposition \ref{proposition_modular_properties}\textnormal{(iii), (iv)} to the left-hand side of \eqref{prop_6} and Theorem 13.17 of Hewitt-Stromberg \cite[p.\,196]{Hewitt-Stromberg-1965} along with Proposition \ref{proposition_embeddings}\textnormal{(ii)} to the right-hand side of \eqref{prop_6} in order to get
	\begin{align*}
		\min\l\{\|u\|^p,\|u\|^q\r\} \leq c_3 \|u\|^{1-\kappa}
	\end{align*}
	for some constant $c_3>0$. As $0<1-\kappa<1<p<q$ this implies \begin{align}\label{prop_7}
		\|u\| \leq c_4
	\end{align}
	for some $c_4>0$. However, from equation \eqref{prop_4} we deduce
	\begin{align}\label{prop_7a}
		\min\l\{\|u\|^p,\|u\|^q\r\} \leq \lambda c_5 \|u\|^{q_1}
	\end{align}
	for some $c_5>0$ where we have used Propositions \ref{proposition_modular_properties}\textnormal{(iii), (iv)} and \ref{proposition_embeddings}\textnormal{(ii)} . From \eqref{prop_7a} we then conclude that
	\begin{align*}
		\|u\| \geq \l(\frac{1}{\lambda c_5}\r)^{\frac{1}{q_1-p}}
		\quad\text{or}\quad
		\|u\| \geq \l(\frac{1}{\lambda c_5}\r)^{\frac{1}{q_1-q}}.
	\end{align*}
	Letting $\lambda \to 0^+$, it follows $\|u\|\to +\infty$ since $p<q<q_1$ contradicting \eqref{prop_7}. This proves the emptiness of $\mathcal{N}_\lambda^\circ$ for all $\lambda\in(0,\hat{\lambda})$.
\end{proof}

Next we are going to prove the nonemptiness of the sets $\mathcal{N}_\lambda^{\pm}$ for small values of the parameter $\lambda>0$ and we will show that the functional $\Theta_\lambda$ achieves its global minimum restricted to the set $\mathcal{N}_\lambda^+$.

\begin{proposition}\label{proposition_nonemptiness_and_existence}
	If hypotheses \textnormal{(H)}  hold, then there exists $\tilde{\lambda}\in (0,\hat{\lambda}]$ such that $\mathcal{N}^{\pm}_\lambda\neq \emptyset$ for all $\lambda \in (0,\tilde{\lambda})$ and for any $\lambda \in (0,\tilde{\lambda})$ there exists $u_\lambda\in\mathcal{N}_\lambda^+$ such that $\Theta_\lambda(u_\lambda)=m_\lambda^+<0$ and $u_\lambda(x) \geq 0$ for a.\,a.\,$x\in\Omega$.
\end{proposition}

\begin{proof}
	Let $u\in \WH$ be such that $u\not\equiv 0$. We consider the function $\tilde{\eta}_u\colon (0,+\infty) \to \R$ defined by
	\begin{align}\label{prop_9c}
	\tilde{\eta}_u(t)= t^{p-q_1}\|u\|_{1,p}^p-t^{-q_1-\kappa+1}\into\zeta(x)|u|^{1-\kappa} \diff x.
	\end{align}
	Recall that  $q_1-p<q_1+\kappa-1$. Hence, there exists a unique $\tilde{t}_u^\circ>0$ such that
	\begin{align*}
		\tilde{\eta}_u\l(\tilde{t}^\circ_u\r)=\max_{t>0} \tilde{\eta}_u(t).
	\end{align*}
	This means
	\begin{align*}
		(p-q_1)\l(\tilde{t}_u^\circ\r)^{p-q_1-1}\|u\|_{1,p}^p+(q_1+\kappa-1)\l(\tilde{t}_u^\circ\r)^{-q_1-\kappa}\into\zeta(x)|u|^{1-\kappa} \diff x=0
	\end{align*}
	and so
	\begin{align}\label{prop_9b}
		\tilde{t}^\circ_u=\l[\frac{(q_1+\kappa-1)\displaystyle \into\zeta(x)|u|^{1-\kappa} \diff x}{(q_1-p)\|u\|_{1,p}^p}\r]^{\frac{1}{p+\kappa-1}}.
	\end{align}
	Using \eqref{prop_9b} into the definition of $\tilde{\eta}_u$ in \eqref{prop_9c} gives
	\begin{align}\label{prop_40}
		\begin{split}
			\tilde{\eta}_u\l(\tilde{t}^\circ_u\r)
			&=\frac{\Big[(q_1-p) \|u\|_{1,p}^p \Big]^{\frac{q_1-p}{p+\kappa-1}}}{\l[(q_1+\kappa-1)\displaystyle \into\zeta(x)|u|^{1-\kappa}\diff x \r]^{\frac{q_1-p}{p+\kappa-1}}}\|u\|_{1,p}^p\\
			&\quad -\frac{\Big[(q_1-p) \|u\|_{1,p}^p \Big]^{\frac{q_1+\kappa-1}{p+\kappa-1}}}{\l[(q_1+\kappa-1)\displaystyle \into\zeta(x)|u|^{1-\kappa}\diff x \r]^{\frac{q_1+\kappa-1}{p+\kappa-1}}}\displaystyle \into\zeta(x)|u|^{1-\kappa}\diff x\\
			&=\frac{(q_1-p)^{\frac{q_1-p}{p+\kappa-1}} \| u\|_{1,p}^{\frac{p(q_1+\kappa-1)}{p+\kappa-1}}}{(q_1+\kappa-1)^{\frac{q_1-p}{p+\kappa-1}}\l[\displaystyle \into\zeta(x)|u|^{1-\kappa}\diff x \r]^{\frac{q_1-p}{p+\kappa-1}}}\\
			&\quad -\frac{(q_1-p)^{\frac{q_1+\kappa-1}{p+\kappa-1}} \| u\|_{1,p}^{\frac{p(q_1+\kappa-1)}{p+\kappa-1}}}{(q_1+\kappa-1)^{\frac{q_1+\kappa-1}{p+\kappa-1}}\l[\displaystyle \into\zeta(x)|u|^{1-\kappa}\diff x \r]^{\frac{q_1-p}{p+\kappa-1}}}\\
			&=\frac{p+\kappa-1}{q_1-p} \l[\frac{q_1-p}{q_1+\kappa-1}\r]^{\frac{q_1+\kappa-1}{p+\kappa-1}}\frac{\| u\|_{1,p}^{\frac{p(q_1+\kappa-1)}{p+\kappa-1}}}{\l[\displaystyle \into\zeta(x)|u|^{1-\kappa}\diff x \r]^{\frac{q_1-p}{p+\kappa-1}}}.
		\end{split}	
	\end{align}
	Applying Theorem 13.17 of Hewitt-Stromberg \cite[p.\,196]{Hewitt-Stromberg-1965} we have 
	\begin{align}\label{prop_42}
		\into\zeta(x)|u|^{1-\kappa}\diff x \leq c_6 \|u\|_{p^*}^{1-\kappa}
	\end{align}
	for some $c_6>0$. Furthermore, we denote by $S$ the best constant of the continuous embedding $\Wp{p}\hookrightarrow \Lp{p^*}$, that is,
	\begin{align}\label{prop_41}
		S\|u\|_{p^*}^p \leq \|u\|_{1,p}^p.
	\end{align}
	From \eqref{prop_40}, \eqref{prop_42} and \eqref{prop_41} it follows that
	\begin{align}\label{prop_41c}
		\begin{split}
		&\tilde{\eta}_u\l(\tilde{t}^\circ_u\r)-\lambda \|u\|_{q_1}^{q_1}\\
		&=\frac{p+\kappa-1}{q_1-p} \l[\frac{q_1-p}{q_1+\kappa-1}\r]^{\frac{q_1+\kappa-1}{p+\kappa-1}}\frac{\| u\|_{1,p}^{\frac{p(q_1+\kappa-1)}{p+\kappa-1}}}{\l[\displaystyle \into\zeta(x)|u|^{1-\kappa}\diff x \r]^{\frac{q_1-p}{p+\kappa-1}}}-\lambda \|u\|_{q_1}^{q_1}\\
		&\geq \frac{p+\kappa-1}{q_1-p} \l[\frac{q_1-p}{q_1+\kappa-1}\r]^{\frac{q_1+\kappa-1}{p+\kappa-1}}\frac{S^{\frac{q_1+\kappa-1}{p+\kappa-1}}\l(\| u\|_{p^*}^p\r)^{\frac{q_1+\kappa-1}{p+\kappa-1}}}{\l(c_6 \|u\|_{p^*}^{1-\kappa}\r)^{\frac{q_1-p}{p+\kappa-1}}}-\lambda c_7\|u\|_{p^*}^{q_1}\\
		&= \Big [c_8-\lambda c_7\Big] \|u\|_{p^*}^{q_1}
		\end{split}
	\end{align}
	for some constants $c_7, c_8>0$. From \eqref{prop_41c} we conclude that there exists  $\tilde{\lambda} \in (0,\hat{\lambda}]$ independent of $u$ such that
	\begin{align}\label{prop_43}
		\tilde{\eta}_u\l(\tilde{t}^\circ_u\r)-\lambda \|u\|_{q_1}^{q_1}>0 \quad\text{for all }\lambda \in \l(0,\tilde{\lambda}\r).
	\end{align}
	
	Now we introduce the function $\eta_u\colon (0,+\infty) \to \R$ defined by
	\begin{align*}
		\eta_u(t)= t^{p-q_1}\|u\|_{1,p}^p+t^{q-q_1}\|\nabla u\|_{q,\mu}^q+t^{p_* -q_1}\| u\|_{p_* ,\beta,\partial\Omega}^{p_* }-t^{-q_1-\kappa+1}\into\zeta(x)|u|^{1-\kappa} \diff x.
	\end{align*}
	Note that the equation 
	\begin{align}\label{prop_7b}
		\begin{split}
			0 = 
			\eta'_u(t)
			&=(p-q_1)t^{p-q_1-1}\|u\|_{1,p}^p+(q-q_1)t^{q-q_1-1}\|\nabla u\|_{q,\mu}^q \\
			&\quad +(p_* -q_1)t^{p_* -q_1-1}\|u\|_{p_* ,\beta,\partial\Omega}^{p_* }  -(-q_1-\kappa+1)t^{-q_1-\kappa}\into\zeta(x)|u|^{1-\kappa}\diff x
		\end{split}
	\end{align}
	is equivalent to
	\begin{align}\label{prop_7c}
		\begin{split}
			& (q_1-p)t^{p+\kappa-1}\|u\|_{1,p}^p
			+(q_1-q)t^{q+\kappa-1}\|\nabla u\|_{q,\mu}^q 
			+(q_1-p_*)t^{p_* +\kappa -1}\|u\|_{p_* ,\beta,\partial\Omega}^{p_* } \\
			& = (q_1+\kappa-1)\into\zeta(x)|u|^{1-\kappa}\diff x.
		\end{split}
	\end{align}
	Since $0<q_1-q<q_1-p<q_1+\kappa-1$, $0<q_1-p_*$, $0 < p + \kappa - 1 < q + \kappa - 1$ and $0<p_* + \kappa - 1$, we can observe that the left-hand side of \eqref{prop_7c}, which we shall denote as $\xi_u (t)$, fulfills
	\begin{equation*}
		\lim_{t \to 0^+} \xi_u (t) = 0, \quad
		\lim_{t \to +\infty} \xi_u (t) = +\infty \quad\text{and}\quad 
		\xi_u' (t) > 0 \text{ for all } t>0.
	\end{equation*}
	From the two limits and via the intermediate value theorem one can derive that there exists $t^\circ_u > 0$ such that \eqref{prop_7c} holds, and from the remaining claim that this value is unique due to the injectivity of $\xi_u (t)$. Furthermore, if we consider $\eta'_u(t) > 0$ ($\eta'_u(t) < 0$), then \eqref{prop_7c} holds with a sign $<$ ($>$), and as $\xi_u (t)$ is strictly increasing this holds for $t<t^\circ_u$ ($t>t^\circ_u$). Hence $\eta_u(t)$ is strictly increasing in $(0,t^\circ_u)$, strictly decreasing in  $(t^\circ_u,\infty)$ and
	\begin{align*}
		\eta_u\l(t^\circ_u\r)=\max_{t>0} \eta_u(t).
	\end{align*}
	It is easy to see that $\eta_u \geq \hat{\eta}_u$. Thus, from \eqref{prop_43} there exists  $\tilde{\lambda} \in (0,\hat{\lambda}]$ independent of $u$ such that
	\begin{align*}
		\eta_u\l(t^\circ_u\r)-\lambda \|u\|_{q_1}^{q_1}>0 \quad\text{for all }\lambda \in \l(0,\tilde{\lambda}\r).
	\end{align*}
	On the other hand, as $0<q_1-q<q_1-p<q_1+\kappa-1$, $0<q_1-p_*$ there holds
	\begin{equation*}
		\lim\limits_{t \to 0^+} \eta_u (t) = - \infty \quad\text{and}\quad
		\lim\limits_{t \to +\infty} \eta_u (t) = 0.
	\end{equation*}
	By the intermediate value theorem and the injectivity of $\eta_u(t)$ in $(0,t^\circ_u)$ and $(t^\circ_u,\infty)$, there exist unique numbers $t^1_u<t^\circ_u<t^2_u$ such that
	\begin{equation}\label{prop_8}
		\eta_u\l(t^1_u\r)=\lambda \|u\|_{q_1}^{q_1}=\eta_u\l(t^2_u\r)
		\quad \text{and}\quad 
		\eta'_u\l(t^2_u\r)<0<\eta'_u\l(t^1_u\r).
	\end{equation}

	Recall that the fibering function $\psi_u\colon[0,+\infty)\to \R$ is given by 
	\begin{align*}
		\psi_u(t)=\Theta_\lambda (tu)\quad\text{for all }t\geq 0.
	\end{align*}
	We have
	\begin{align*}
		\psi'_u\l(t^1_u\r)&=\l(t^1_u\r)^{p-1}\|u\|_{1,p}^p+\l(t^1_u\r)^{q-1}\|\nabla u\|_{q,\mu}^q+\l(t^1_u\r)^{p_* -1}\| u\|_{p_* ,\beta,\partial\Omega}^{p_* }\\ 
		&\quad -\l(t^1_u\r)^{-\kappa}\into\zeta(x)|u|^{1-\kappa}\diff x-\lambda \l(t^1_u\r)^{q_1-1}\|u\|_{q_1}^{q_1}
	\end{align*}
	and
	\begin{align}\label{prop_9}
		\begin{split}
		\psi''_u\l(t^1_u\r)&=(p-1)\l(t^1_u\r)^{p-2}\|u\|_{1,p}^p+(q-1)\l(t^1_u\r)^{q-2}\|\nabla u\|_{q,\mu}^q\\
		&\quad+(p_* -1)\l(t^1_u\r)^{p_* -2}\| u\|_{p_* ,\beta,\partial\Omega}^{p_* } +\kappa \l(t^1_u\r)^{-\kappa-1}\into\zeta(x)|u|^{1-\kappa}\diff x\\
		&\quad -\lambda(q_1-1) \l(t^1_u\r)^{q_1-2}\|u\|_{q_1}^{q_1}.
		\end{split}
	\end{align}

	The first relation in \eqref{prop_8} gives
	\begin{align}\label{prop_9f}
		\begin{split}
		&\l(t^1_u\r)^{p-q_1}\|u\|_{1,p}^p +\l(t^1_u\r)^{q-q_1}\|\nabla u\|_{q,\mu}^q+\l(t^1_u\r)^{p_* -q_1}\| u\|_{p_* ,\beta,\partial\Omega}^{p_* }\\
		&-\l(t^1_u\r)^{-q_1-\kappa+1}\into\zeta(x)|u|^{1-\kappa}\diff x=\lambda \|u\|_{q_1}^{q_1}.
		\end{split}
	\end{align}
	Now we multiply \eqref{prop_9f} with $\kappa \l(t^1_u\r)^{q_1-2}$ and $-(q_1-1)\l(t^1_u\r)^{q_1-2}$, respectively. It follows
	\begin{align}\label{prop_10}
		\begin{split}
		&\kappa \l(t^1_u\r)^{p-2}\|u\|_{1,p}^p +\kappa \l(t^1_u\r)^{q-2}\|\nabla u\|_{q,\mu}^q+\kappa \l(t^1_u\r)^{p_* -2}\|u\|_{p_* ,\beta,\partial\Omega}^{p_* }-\kappa\lambda \l(t^1_u\r)^{q_1-2} \|u\|_{q_1}^{q_1}\\
		&= \kappa \l(t^1_u\r)^{-\kappa-1}\into\zeta(x)|u|^{1-\kappa}\diff x
	\end{split}	
	\end{align}
	and
	\begin{align}\label{prop_11}
		\begin{split}
			&-(q_1-1)\l(t^1_u\r)^{p-2}\|u\|_{1,p}^p -(q_1-1)\l(t^1_u\r)^{q-2}\|\nabla u\|_{q,\mu}^q-(q_1-1)\l(t^1_u\r)^{p_* -2}\|u\|_{p_* ,\beta,\partial\Omega}^{p_* }\\
			&\quad +(q_1-1)\l(t^1_u\r)^{-\kappa-1}\into\zeta(x)|u|^{1-\kappa}\diff x =-\lambda(q_1-1) \l(t^1_u\r)^{q_1-2} \|u\|_{q_1}^{q_1}.
		\end{split}	
	\end{align}
	
	Now we use \eqref{prop_10}  in \eqref{prop_9} which leads to
	\begin{align}\label{prop_12}
		\begin{split}
			\psi''_u\l(t^1_u\r)&=(p+\kappa -1)\l(t^1_u\r)^{p-2}\|u\|_{1,p}^p+(q+\kappa-1)\l(t^1_u\r)^{q-2}\|\nabla u\|_{q,\mu}^q\\ 
			&\quad +(p_* +\kappa-1)\l(t^1_u\r)^{p_* -2}\| u\|_{p_* ,\beta,\partial\Omega}^{p_* }-\lambda(q_1+\kappa -1) \l(t^1_u\r)^{q_1-2}\|u\|_{q_1}^{q_1}\\
			&= \l(t^1_u\r)^{-2}\Big[(p+\kappa -1)\l(t^1_u\r)^{p}\|u\|_{1,p}^p+(q+\kappa-1)\l(t^1_u\r)^{q}\|\nabla u\|_{q,\mu}^q\\ 
			&\qquad\qquad\quad +(p_* +\kappa-1)\l(t^1_u\r)^{p_* }\| u\|_{p_* ,\beta,\partial\Omega}^{p_* }-\lambda(q_1+\kappa -1) \l(t^1_u\r)^{q_1}\|u\|_{q_1}^{q_1}\Big].
		\end{split}
	\end{align}

	On the other hand, applying \eqref{prop_11} in \eqref{prop_9} along with the representation in \eqref{prop_7b} yields
	\begin{equation}\label{prop_13}
		\begin{split}
			\psi''_u\l(t^1_u\r)&=(p-q_1)\l(t^1_u\r)^{p-2}\|u\|_{1,p}^p+(q-q_1)\l(t^1_u\r)^{q-2}\|\nabla u\|_{q,\mu}^q\\
			& \quad +(p_* -q_1)\l(t^1_u\r)^{p_* -2}\| u\|_{p_* ,\beta,\partial\Omega}^{p_* }+(q_1+\kappa-1) \l(t^1_u\r)^{-\kappa-1}\into\zeta(x)|u|^{1-\kappa}\diff x\\
			&=\l(t^1_u\r)^{q_1-1} \eta'_u\l(t^1_u\r)>0,
		\end{split}
	\end{equation}
	see \eqref{prop_7b}. From \eqref{prop_12} and \eqref{prop_13} we conclude that
	\begin{align*}
		\begin{split}
			&(p+\kappa -1)\l(t^1_u\r)^{p}\|u\|_{1,p}^p+(q+\kappa-1)\l(t^1_u\r)^{q}\|\nabla u\|_{q,\mu}^q\\ 
			&\qquad\qquad\quad +(p_* +\kappa-1)\l(t^1_u\r)^{p_* }\| u\|_{p_* ,\beta,\partial\Omega}^{p_* }-\lambda(q_1+\kappa -1) \l(t^1_u\r)^{q_1}\|u\|_{q_1}^{q_1}>0,
		\end{split}
	\end{align*}
	since $t^1_u>0$. This shows that
	\begin{equation*}
		t^1_u u\in \mathcal{N}_\lambda^+ \quad \text{for all } \lambda\in \l(0,\tilde{\lambda}\r].
	\end{equation*}
	Hence, $\mathcal{N}_\lambda^+\neq \emptyset$. The same treatment can be done for the point $t^2_u$ in order to prove that $\mathcal{N}_\lambda^-\neq \emptyset$. 
	
	Now we are going to show the second assertion of the proposition. To this end, let $\{u_n\}_{n\in\N}\subset \mathcal{N}_\lambda^+$ be a minimizing sequence, that is,
	\begin{align}\label{prop_14}
		\Theta_\lambda(u_n) \searrow m^+_\lambda <0 \quad\text{as }n\to\infty.
	\end{align}
	First, we know that $\{u_n\}_{n\in\N}\subset \WH$ is bounded since $\mathcal{N}_\lambda^+ \subset \mathcal{N}_\lambda$ and by applying Proposition \ref{proposition_coerivity}. Hence, we may assume that
	\begin{align}\label{prop_15}
		u_n\weak u_\lambda \quad\text{in }\WH 
		\quad
		u_n\to u_\lambda \quad\text{in }\Lp{q_1}\quad\text{and}\quad u_n\weak u_\lambda\quad\text{in }\Lprand{p_* },
	\end{align}
	see Proposition \ref{proposition_embeddings}\textnormal{(ii)}, \textnormal{(iii)}. From \eqref{prop_14} and \eqref{prop_15} it follows that
	\begin{align*}
		\Theta_\lambda(u_\lambda) \leq \liminf_{n\to+\infty} \Theta_\lambda(u_n)<0=\Theta_\lambda(0).
	\end{align*}
	Therefore, $u_\lambda\neq 0$.
	
	Next, we want to prove that
	\begin{align}\label{prop_15g}
		\lim_{n\to+\infty} \rho(u_n)=\rho(u_\lambda)
	\end{align}
	for a subsequence (still denoted by $u_n$).
	
	{\bf Claim 1:} $\liminf_{n\to+\infty} \|u_n\|_{1,p}^p=\|u_\lambda\|_{1,p}^p$
	
	Let us suppose Claim 1 is not true. Then we have 
	\begin{align*}
		\liminf_{n\to+\infty} \|u_n\|_{1,p}^p>\|u_\lambda\|_{1,p}^p.
	\end{align*}
	Using this and \eqref{prop_8} along with the weak lower semicontinuity of the corresponding norms and seminorms results in
	\begin{align*}
		\begin{split}
			&\liminf_{n\to+\infty} \psi'_{u_n}\l(t^1_{u_\lambda}\r)\\
			&=\liminf_{n\to+\infty}\l[\l(t^1_{u_\lambda}\r)^{p-1}\| u_n\|_{1,p}^p+\l(t^1_{u_\lambda}\r)^{q-1}\|\nabla u_n\|_{q,\mu}^q+\l(t^1_{u_\lambda}\r)^{p_* -1}\|u_n\|_{p_* ,\beta,\partial\Omega}^{p_* } \r.\\
			&\l. \qquad\qquad\quad -\l(t^1_{u_\lambda}\r)^{-\kappa}\into\zeta(x)|u_n|^{1-\kappa}\diff x-\lambda \l(t^1_{u_\lambda}\r)^{q_1-1}\|u_n\|_{q_1}^{q_1}\r]\\
			&>\l(t^1_{u_\lambda}\r)^{p-1}\|u_\lambda\|_{1,p}^p+\l(t^1_{u_\lambda}\r)^{q-1}\|\nabla u_\lambda\|_{q,\mu}^q+\l(t^1_{u_\lambda}\r)^{p_* -1}\|u_\lambda\|_{p_* ,\beta,\partial\Omega}^{p_* } \\
			&\quad -\l(t^1_{u_\lambda}\r)^{-\kappa}\into\zeta(x)|u_\lambda|^{1-\kappa}\diff x-\lambda \l(t^1_{u_\lambda}\r)^{q_1-1}\|u_\lambda\|_{q_1}^{q_1}\\
			&=\psi'_{u_\lambda}\l(t^1_{u_\lambda}\r)=\l(t^1_{u_\lambda}\r)^{q_1-1}\l[\eta_{u_\lambda}\l(t^1_{u_\lambda}\r)-\lambda\|u_\lambda\|_{q_1}^{q_1}\r]=0.
		\end{split}
	\end{align*}
	
	Hence, there exists a number $n_0\in \N$ such that $\psi'_{u_n}(t^1_{u_\lambda})>0$ for all $n>n_0$. We know that $u_n\in \mathcal{N}^+_{\lambda}\subset \mathcal{N}_{\lambda}$ and $\psi'_{u_n}(t)=t^{q_1-1} \l[\eta_{u_n}(t)-\lambda\|u_n\|_{q_1}^{q_1}\r]$. Therefore, we conclude that $\psi'_{u_n}(t)<0$ for all $t\in(0,1)$ and $\psi'_{u_n}(1)=0$ which implies $t^1_{u_\lambda}>1$.
	
	Recall that $\psi_{u_\lambda}$ is decreasing on $(0,t^1_{u_\lambda}]$. This implies 
	\begin{align*}
		\Theta_{\lambda} \l(t^1_{u_\lambda} u_\lambda\r) \leq \Theta_{\lambda}\l(u_\lambda\r)<m^+_{\lambda}.
	\end{align*}
	Since $t^1_{u_\lambda} u_\lambda\in \nc^+_{\lambda}$ we have 
	\begin{align*}
		m^+_{\lambda}\leq \Theta_{\lambda}\l(t^1_{u_\lambda} u_\lambda\r)<m^+_{\lambda},
	\end{align*}
	a contradiction. So Claim 1 is proved.

	From Claim 1 we find a subsequence (still denoted by $u_n$) such that
	\begin{align}\label{conv1}
		 \|u_n\|_{1,p}^p \to  \|u_\lambda\|_{1,p}^p.
	\end{align}
	
	{\bf Claim 2:} $\liminf_{n\to+\infty} \|\nabla u_n\|^q_{q,\mu}=\|\nabla u_\lambda\|_{q,\mu}^q$ for the subsequence in \eqref{conv1}.
	
	As before, let us suppose Claim 2 is not true. So we have 
	\begin{align*}
		\liminf_{n\to+\infty} \|\nabla u_n\|^q_{q,\mu}>\|\nabla u_\lambda\|_{q,\mu}^q.
	\end{align*}
	
	Then we can argue exactly as in the proof of Claim 1. This shows Claim 2.
	
	From Claim  2 we find a subsequence (still denoted by $u_n$) such that
	\begin{align}\label{conv2}
		\|\nabla u_n\|_{q,\mu}^q \to  \|\nabla u_\lambda\|_{q,\mu}^q.
	\end{align}
	
	{\bf Claim 3:} $\liminf_{n\to+\infty} \|u_n\|_{p_* ,\beta,\partial\Omega}^{p_* }=\|u_\lambda\|_{p_* ,\beta,\partial\Omega}^{p_* }$ for the subsequence in \eqref{conv2}.
	
	The proof is the same as in Claims 1 and 2. So we find again a subsequence  (still denoted by $u_n$) such that
	\begin{align}\label{conv3}
		\|u_n\|_{p_* ,\beta,\partial\Omega}^{p_* } \to  \|u_\lambda\|_{p_* ,\beta,\partial\Omega}^{p_* }.
	\end{align}
	
	For the sequence in \eqref{conv3} we know that the convergences in \eqref{conv1} and \eqref{conv2} hold true. So combining \eqref{conv1}--\eqref{conv3} for this sequence we see that \eqref{prop_15g} is satisfied. Since the integrand corresponding to the modular function $\rho(\cdot)$ is uniformly convex, this implies $\rho (\frac{u_n-u_\lambda}{2})\to 0$. Then Proposition \ref{proposition_modular_properties}\textnormal{(v)} implies that $u_n \to u_\lambda$ in $\WH$, see also Fan-Guan \cite[Theorems 3.2 and 3.5]{Fan-Guan-2010}. By the continuity of $\Theta_\lambda$ we have $\Theta_{\lambda}(u_n)\to \Theta_{\lambda}(u_\lambda)$ and thus, $\Theta_{\lambda}(u_\lambda)=m^+_{\lambda}$. We know that $u_n\in \nc^+_{\lambda}$ for all $n\in \N$, that is,
	\begin{align}\label{prop_15j}
		\begin{split}
		&(p+\kappa-1)\|u_n\|_{1,p}^p+(q+\kappa-1)\|\nabla u_n\|_{q,\mu}^q+(p_* +\kappa-1)\| u_n\|_{p_* ,\beta,\partial\Omega}^{p_* }\\
		& -\lambda (q_1+\kappa-1) \|u_n\|_{q_1}^{q_1}>0.
		\end{split}	
	\end{align}
	Passing to the limit in \eqref{prop_15j} as $n\to+\infty$ we obtain
	\begin{equation}\label{prop_16}
		(p+\kappa-1)\|u_\lambda\|_{1,p}^p+(q+\kappa-1)\|\nabla u_\lambda\|_{q,\mu}^q+(p_* +\kappa-1)\| u_\lambda\|_{p_* ,\beta,\partial\Omega}^{p_* }-\lambda (q_1+\kappa-1) \|u_\lambda\|_{q_1}^{q_1}\geq 0.
	\end{equation}
	Since $\lambda \in (0,\tilde{\lambda})$ and $\tilde{\lambda}\leq \hat{\lambda}$, we conclude from Proposition \ref{proposition_emptiness} that we have a strict inequality in \eqref{prop_16}. Hence, $u_\lambda\in \mathcal{N}^+_{\lambda}$. Note that we can always use $|u_\lambda|$ instead of $u_\lambda$, so we may assume that $u_\lambda(x)\geq 0$ for a.\,a.\,$x\in\Omega$ such that $u_\lambda\neq 0$.
\end{proof}

The next proposition plays a key role in order to prove that $u_\lambda$ is a weak solution of problem \eqref{problem}.

\begin{proposition}\label{proposition_energy_estimate}
		If hypotheses \textnormal{(H)}  hold, $h\in\WH$ and $\lambda \in(0, \tilde{\lambda}]$, then we can find $\delta>0$ such that $\Theta_{\lambda}(u_\lambda)\leq \Theta_{\lambda}(u_\lambda+th)$ for all $t\in [0,\delta]$.
\end{proposition}

\begin{proof}
	Taking $u\in\mathcal{N}_\lambda^+$, we introduce the function $\ph\colon\WH\times (0,\infty)\to\R$ defined by
	\begin{align*}
		\ph(y,t)
		&=t^{p+\kappa-1}\|u+y\|_{1,p}^p+t^{q+\kappa-1}\|\nabla (u+y)\|_{q,\mu}^q+t^{p_* +\kappa-1}\| u+y\|_{p_* ,\beta,\partial\Omega}^{p_* }\\
		& \quad -\into\zeta(x)|u+y|^{1-\kappa}\diff x
		-\lambda t^{q_1+\kappa-1} \|u+y\|_{q_1}^{q_1}\quad \text{for all } y\in \WH.
	\end{align*}
	Due to $u\in \mathcal{N}^+_\lambda \subset \mathcal{N}_\lambda$, we easily see that $\ph(0,1)=0$. Moreover, since $u \in \mathcal{N}^+_\lambda$, we obtain
	\begin{align*}
		\ph'_t(0,1)&=(p+\kappa-1)\|u\|_{1,p}^p+(q+\kappa-1)\|\nabla u\|_{q,\mu}^q+(p_* +\kappa-1)\|u\|_{p_* ,\beta,\partial\Omega}^{p_* } \\
		&\quad -\lambda (q_1+\kappa-1)\|u\|_{q_1}^{q_1}>0.
	\end{align*}
	Now we can apply the implicit function theorem, see, for example, Berger \cite[p.\,115]{Berger-1977}, which ensures the existence of $\eps>0$ and a continuous function $\chi\colon B_\eps(0)\to (0,\infty)$ such that
	\begin{align*}
		\chi(0)=1\quad\text{and}\quad \chi(y)(u+y) \in \mathcal{N}_\lambda \quad\text{for all } y\in B_\eps(0),
	\end{align*}
	where $B_\eps(0)=\l\{u\in\WH\,:\, \|u\|<\eps\r\}$. Choosing $\eps>0$ small enough, we have 
	\begin{align}\label{prop_17g}
		\chi(0)=1\quad\text{and}\quad \chi(y)(u+y) \in \mathcal{N}^{+}_\lambda \quad\text{for all }y\in B_\eps(0).
	\end{align}

	We define now the function $\Xi_h\colon [0,+\infty)\to \R$ given by
	\begin{align}\label{prop_17}
		\begin{split}
			\Xi_h(t)
			&=(p-1)\l \|u_\lambda+th\r\|_{1,p}^p+(q-1)\|\nabla u_\lambda+t\nabla h\|_{q,\mu}^q+(p_* -1)\| u_\lambda+t h\|_{p_* ,\beta,\partial\Omega}^{p_* }\\
			& \quad+\kappa \into\zeta(x)\l|u_\lambda+th \r|^{1-\kappa}\diff x-\lambda (q_1-1)\l\| u_\lambda+th \r\|_{q_1}^{q_1}.
		\end{split}
	\end{align}
	We know that $u_\lambda \in \mathcal{N}_\lambda$ and $u_\lambda \in \mathcal{N}_\lambda^+$. This gives
	\begin{align}\label{prop_18}
		\kappa \into\zeta(x)\l| u_\lambda\r|^{1-\kappa} \diff x=\kappa \l\|  u_\lambda\r\|_{1,p}^p+\kappa \l\|\nabla u_\lambda\r\|_{q,\mu}^q+\kappa \l\| u_\lambda\r\|_{p_* ,\beta,\partial\Omega}^{p_* }-\lambda \kappa \l\| u_\lambda\r\|_{q_1}^{q_1}
	\end{align}
	and
	\begin{align}\label{prop_19}
		\begin{split}
		&(p+\kappa-1)\l\| u_\lambda\r\|_{1,p}^p+(q+\kappa-1) \l\|\nabla u_\lambda\r\|_{q,\mu}^q+(p_* +\kappa-1) \l\| u_\lambda\r\|_{p_* ,\beta,\partial\Omega}^{p_* }\\
		& -\lambda(q_1+\kappa-1) \l\| u_\lambda\r\|_{q_1}^{q_1}>0.
		\end{split}	
	\end{align}
	Using  \eqref{prop_18} and \eqref{prop_19} in \eqref{prop_17} we infer that $\Xi_h(0)>0$ and since $\Xi_h\colon [0,+\infty)\to \R$ is continuous there exists a number $\delta_0>0$ such that
	\begin{align*}
		\Xi_h(t)>0 \quad\text{for all }t \in [0,\delta_0].
	\end{align*}
	From the first part of the proof, see \eqref{prop_17g}, we know that for every $t \in [0,\delta_0]$ there exists $\chi(t)>0$ such that
	\begin{align*}
		\chi(t)\l(u_\lambda+th\r)\in \mathcal{N}_\lambda^+
		\quad\text{and}\quad
		\chi(t) \to 1 \quad\text{as } t\to 0^+.
	\end{align*}
	From Proposition \ref{proposition_nonemptiness_and_existence} it follows that
	\begin{align}\label{prop_20}
		m_\lambda^+=\Theta_\lambda \l(u_\lambda\r) \leq \Theta_\lambda \l(\chi(t)\l(u_\lambda+th\r)\r)\quad\text{for all } t \in [0,\delta_0].
	\end{align}

	Let $u \in \WH$ be arbitrary and recall that $\psi_u(t)=\Theta_\lambda (tu)$ for all $t\geq 0$. Note that for $t>0$, we have $t \psi_u'(t)=0$ if and only if $tu \in \mathcal{N}_\lambda$, that is, $t=t^1_u$ or $t=t^2_u$ by \eqref{prop_8}. From \eqref{prop_13} we also know that $\psi_u''(t^1_u)>0>\psi_u''(t^2_u)$. On the other hand, there exists a unique $t^*_u>0$ such that $\psi_u''(t^*_u)=0$. Hence, $t^1_u<t^*_u<t^2_u$. Altogether, if $\psi_u''(t)>0$, then $t\in (0,t^*_u) \subseteq (0, t^2_u)$ and $\psi_u(t)\geq \psi_u(t^1_u)$.
	
	\begin{center}
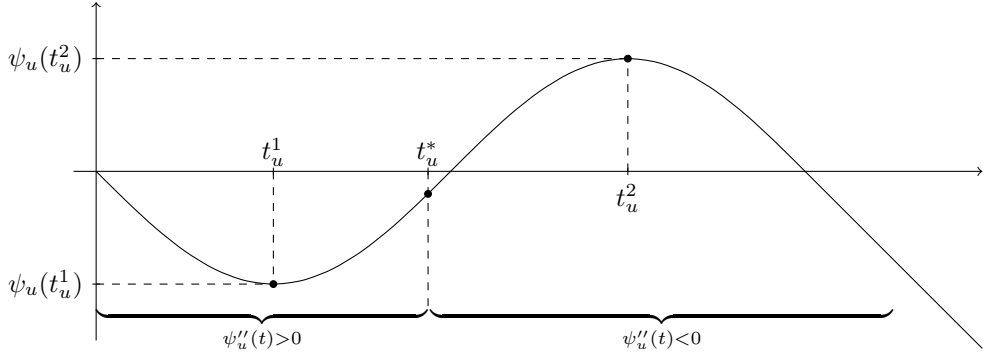

		\begin{tikzpicture}[scale=1.5]
			\draw[->] (-0.2,1) -- (2.5*pi,1);
			\draw[->] (0,-0.5) -- (0,2.5);
			\draw[domain=0:2*pi,smooth,variable=\x] plot 
			({\x},{-2*sin(\x r +0 r)/2+1});
			\draw[domain=2*pi:2.5*pi,smooth,variable=\x] plot 
			({\x},{-\x+2*pi+1});
			\draw (pi/2,\xticklength/1.5+1) -- (pi/2,-\xticklength/1.5+1) node[above] {$t^1_u$};
			\draw[dashed] (pi/2,1) -- (pi/2,0);
			\draw[fill=black] (pi/2,0) circle (0.03cm);
			\draw (3*pi/2,\xticklength/1.5+1) -- (3*pi/2,-\xticklength/1.5+1) node[below] {$t^2_u$};
			\draw[dashed] (3*pi/2,1) -- (3*pi/2,2);
			\draw[fill=black] (3*pi/2,2) circle (0.03cm);
			\draw (\xticklength/1.5,2) -- (-\xticklength/1.5,2) node[left] {$\psi_u(t^2_u)$};
			\draw[dashed] (0,2) -- (3*pi/2,2);
			\draw (\xticklength/1.5,0) -- (-\xticklength/1.5,0) node[left] {$\psi_u(t^1_u)$};
			\draw[dashed] (0,0) -- (pi/2,0);
			\draw (pi-0.2,\xticklength/1.5+1) -- (pi-0.2,-\xticklength/1.5+1) node[above] {$t^*_u$};
			\draw[dashed] (pi-0.2,1) -- (pi-0.2,-0.2);
			\draw[fill=black] (pi-0.2,0.8) circle (0.03cm);
			\node at (pi/2-0.1,-0.4) {$\underbrace{\hspace{4.4cm}}_{\psi_u''(t)>0}$};
			\node at (3*pi/2+0.3,-0.4) {$\underbrace{\hspace{6.15cm}}_{\psi_u''(t)<0}$};
		\end{tikzpicture}
		\captionof{figure}{The graph of $\psi_u(t)$.}\label{figure_1}
	\end{center}

	Now, from $\psi_{u_\lambda}''(1)>0$ and the continuity in $t$, we have $\psi''_{u_\lambda +th}(1)>0$ for $t \in [0,\delta]$ with $\delta\in (0,\delta_0]$. Therefore, using \eqref{prop_20}
	\begin{align*}
		m_\lambda^+
		=\Theta_\lambda \l(u_\lambda\r) 
		\leq \Theta_\lambda \l(\chi(t)\l(u_\lambda+th\r)\r)
		=\psi_{u_\lambda+th}(\chi(t))
		\leq \psi_{u_\lambda +th}(1)
		=\Theta_\lambda \l(u_\lambda+th\r)
	\end{align*}
	for all $t \in [0,\delta]$.
\end{proof}

Now we are ready to show that $u_\lambda$ is indeed a weak solution for problem \eqref{problem}.

\begin{proposition}\label{proposition_first_weak_solution}
	If hypotheses \textnormal{(H)}  hold and  $\lambda \in(0, \tilde{\lambda})$, then $u_\lambda$ is a weak solution of problem \eqref{problem} such that $\Theta_\lambda(u_\lambda)<0$.
\end{proposition}

\begin{proof}
	First, we mention that $u_\lambda\geq 0$ for a.\,a.\,$x\in\Omega$ and $\Theta_{\lambda}(u_\lambda)<0$, see Proposition \ref{proposition_nonemptiness_and_existence}.
	
	{\bf Claim:} $u_\lambda> 0$ for a.\,a.\,$x\in\Omega$
	
	Arguing indirectly and suppose there exists a set $B$ with positive measure such that $u_\lambda=0$ in $B$. Let $h\in \WH$ with $h > 0$ and let $t\in (0,\delta)$ (see Proposition \ref{proposition_energy_estimate}), then $(u_\lambda+th)^{1-\kappa}>u_\lambda^{1-\kappa}$ a.\,e.\,in $\Omega\setminus B$. Applying Proposition \ref{proposition_energy_estimate} yields
	\begin{align*}
		0
		&\leq \frac{\Theta_\lambda(u_\lambda+th)-\Theta_\lambda(u_\lambda)}{t} \\
		&=\frac{1}{p} \frac{\|u_\lambda+th\|_{1,p}^p-\|u_\lambda\|_{1,p}^p}{t} +\frac{1}{q} \frac{\|\nabla (u_\lambda+th)\|_{q,\mu}^q-\|\nabla u_\lambda\|_{q,\mu}^q}{t}\\
		& \quad +\frac{1}{p_* } \frac{\|u_\lambda+th\|_{p_* ,\beta,\partial\Omega}^{p_* }-\|u_\lambda\|_{p_* ,\beta,\partial\Omega}^{p_* }}{t}-\frac{1}{(1-\kappa)t^{\kappa}} \int_B \zeta(x)h^{1-\kappa} \diff x\\
		&\quad - \frac{1}{1-\kappa}\int_{\Omega\setminus B} \zeta(x)\frac{(u_\lambda+th)^{1-\kappa}-(u_\lambda)^{1-\kappa}}{t}\diff x\
		-\frac{\lambda}{q_1}
		\frac{\|u_\lambda+th\|_{q_1}^{q_1}-\|u_\lambda\|_{q_1}^{q_1}}{t}\\
		&<\frac{1}{p} \frac{\|u_\lambda+th\|_{1,p}^p-\|u_\lambda\|_{1,p}^p}{t} +\frac{1}{q} \frac{\|\nabla (u_\lambda+th)\|_{q,\mu}^q-\|\nabla u_\lambda\|_{q,\mu}^q}{t}\\
		&\quad +\frac{1}{p_* } \frac{\|u_\lambda+th\|_{p_* ,\beta,\partial\Omega}^{p_* }-\|u_\lambda\|_{p_* ,\beta,\partial\Omega}^{p_* }}{t}
		-\frac{1}{(1-\kappa)t^{\kappa}} \int_B \zeta(x)h^{1-\kappa} \diff x\\
		&\quad -\frac{\lambda}{q_1}
		 \frac{\|u_\lambda+th\|_{q_1}^{q_1}-\|u_\lambda\|_{q_1}^{q_1}}{t}.
	\end{align*}
	From this we conclude by hypothesis \textnormal{(H)(v)} that
	\begin{align*}
		0
		&\leq \frac{\Theta_\lambda(u_\lambda+th)-\Theta_\lambda(u_\lambda)}{t} \to -\infty \quad \text{as }t\to 0^+,
	\end{align*}
	which is a contradiction. This shows the Claim and so $u_\lambda>0$ a.\,e.\,in $\Omega$.
	
	Let us now show that
	\begin{equation}\label{def1}
		\zeta(\cdot)(u_\lambda)^{-\kappa} h\in \Lp{1}  \quad \text{for all } h\in\WH
	\end{equation}
	and
	\begin{align}\label{def2}
		\begin{split}
			& \into \Big(|\nabla u_\lambda|^{p-2} \nabla u_\lambda+ \mu(x) |\nabla u_\lambda|^{q-2} \nabla u_\lambda\Big) \cdot \nabla h \diff x\\
			& \quad +\into \alpha(x) (u_\lambda)^{p-1}h\diff x+\int_{\partial\Omega}\beta(x)(u_\lambda)^{p_* -1} h\diff \sigma \\
			&\geq \into \zeta(x)(u_\lambda)^{-\kappa} h \diff x+\lambda \into (u_\lambda )^{q_1-1}h\diff x
		\end{split}
	\end{align}
	for all $h\in \WH$ with $h \geq 0$.
	
	Taking $h\in \WH$ with $h\geq 0$ and choosing a decreasing sequence $\{t_n\}_{n \in\N} \subseteq (0,1]$ such that $\displaystyle \lim_{n\to \infty} t_n=0$, we see that for $n \in\N$, the functions
	\begin{align*}
		\omega_n(x)=\zeta(x)\frac{(u_\lambda(x)+t_nh(x))^{1-\kappa}-u_\lambda(x)^{1-\kappa}}{t_n}
	\end{align*}
	are measurable, nonnegative  and we have
	\begin{align*}
		\lim_{n\to \infty} \omega_n(x)=(1-\kappa) \zeta(x)u_\lambda(x)^{-\kappa}h(x)\quad \text{for a.\,a.\,} x\in\Omega.
	\end{align*}
	From Fatou's lemma we obtain
	\begin{equation}\label{fatou}
		\into \zeta(x) \l (u_\lambda\r)^{-\kappa}h\diff x\leq \frac{1}{1-\kappa}\liminf_{n\to\infty}\into \omega_n\diff x.
	\end{equation}
	Applying again Proposition \ref{proposition_energy_estimate} we get for $n\in\N$ sufficiently large that
	\begin{align*}
		0
		&\leq \frac{\Theta_\lambda(u_\lambda+t_nh)-\Theta_\lambda(u_\lambda)}{t_n} \\
		&=\frac{1}{p} \frac{\|u_\lambda+t_nh\|_{1,p}^p-\|u_\lambda\|_{1,p}^p}{t_n} +\frac{1}{q} \frac{\|\nabla (u_\lambda+t_nh)\|_{q,\mu}^q-\|\nabla u_\lambda\|_{q,\mu}^q}{t_n}\\
		&\quad +\frac{1}{p_* } \frac{\|u_\lambda+t_nh\|_{p_* ,\beta,\partial\Omega}^{p_* }-\|u_\lambda\|_{p_* ,\beta,\partial\Omega}^{p_* }}{t_n}- \frac{1}{1-\kappa}\into \omega_n\diff x-\frac{\lambda}{q_1}
		\frac{\|u_\lambda+t_nh\|_{q_1}^{q_1}-\|u_\lambda\|_{q_1}^{q_1}}{t_n}.
	\end{align*}
	Passing to the limit as $n\to \infty$ and applying \eqref{fatou} we obtain \eqref{def1} and 
	\begin{align*}
		\begin{split}
			& \into \zeta(x)(u_\lambda)^{-\kappa} h \diff x\\
			&\leq \into \Big(|\nabla u_\lambda|^{p-2} \nabla u_\lambda+ \mu(x) |\nabla u_\lambda|^{q-2} \nabla u_\lambda\Big) \cdot \nabla h \diff x\\
			& \quad +\into \alpha(x) (u_\lambda)^{p-1}h\diff x+\int_{\partial\Omega}\beta(x)(u_\lambda)^{p_* -1} h\diff \sigma 
			+\lambda \into (u_\lambda )^{q_1-1}h\diff x.
		\end{split}
	\end{align*}
	This shows \eqref{def2}. Note that it is sufficient to prove the integrability in \eqref{def1} for nonnegative test functions $h\in \WH$.
	
	Now, we can prove that $u_\lambda$ is a weak solution of \eqref{problem}. Fur this purpose, let $v\in\WH$ and $\eps>0$. Taking $h=(u_\lambda+\eps v)_+$ as test function in \eqref{def2} and using the fact that $u_\lambda\in \mathcal{N}_\lambda^+\subset \mathcal{N}_\lambda$ such that $u_\lambda\geq 0$, we obtain
	\begin{align*}
			0& \le \into \ykh{\l|\nabla u_\lambda\r|^{p-2}\nabla u_\lambda + \mu(x)\l|\nabla u_\lambda\r|^{q-2}\nabla u_\lambda} \cdot \nabla (u_\lambda+\eps v)_+ \diff x\\
			& \quad +\into \alpha(x)\l(u_\lambda\r)^{p-1}(u_\lambda+\eps v)_+\diff x+\int_{\partial\Omega}\beta(x)\l(u_\lambda\r)^{p_* -1}(u_\lambda+\eps v)_+\diff \sigma\\
			&\quad -\into \l(\zeta(x)\l(u_\lambda\r)^{-\kappa}+\lambda \l(u_\lambda\r)^{q_1-1}\r)(u_\lambda+\eps v)_+\diff x\\
			&=\into\l(\l|\nabla u_\lambda\r|^{p-2}\nabla u_\lambda + \mu(x)\l|\nabla u_\lambda\r|^{q-2}\nabla u_\lambda\r) \cdot \nabla \l (u_\lambda+\eps v\r) \diff x\\
			&\quad - \int_{\l\{ u_\lambda+\eps v< 0\r\}}\l(\l|\nabla u_\lambda\r|^{p-2}\nabla u_\lambda + \mu(x)\l|\nabla u_\lambda\r|^{q-2}\nabla u_\lambda\r) \cdot \nabla \l( u_\lambda+\eps v \r) \diff x\\
			& \quad +\into \alpha(x) \l(u_\lambda\r)^{p-1} \l (u_\lambda+\eps v\r) \diff x-\int_{\l\{ u_\lambda+\eps v< 0\r\}} \alpha(x) \l(u_\lambda\r)^{p-1} \l (u_\lambda+\eps v\r) \diff x\\
			& \quad +\int_{\partial\Omega} \beta(x) \l(u_\lambda\r)^{p_* -1} \l (u_\lambda+\eps v\r) \diff \sigma-\int_{\l\{ u_\lambda+\eps v< 0\r\}} \beta(x) \l(u_\lambda\r)^{p_* -1} \l (u_\lambda+\eps v\r) \diff \sigma\\
			& \quad -\into\l(\zeta(x)\l(u_\lambda\r)^{-\kappa}+\lambda \l(u_\lambda\r)^{q_1-1}\r) \l( u_\lambda+\eps v\r)\diff x\\
			&\quad + \int_{\{ u_\lambda+\e v< 0\}} \l(\zeta(x)\l(u_\lambda\r)^{-\kappa}+\lambda \l(u_\lambda\r)^{q_1-1}\r)\l( u_\lambda+\eps v \r)\diff x\\
			&=\|u_\lambda\|_{1,p}^p+\|\nabla u_\lambda\|_{q,\mu}^q+\| u_\lambda\|_{p_* ,\beta,\partial\Omega}^{p_* }-\into\zeta(x)|u_\lambda|^{1-\kappa}\diff x-\lambda \|u_\lambda\|_{q_1}^{q_1}\\
			&\quad +\eps \into \l ( \l|\nabla u_\lambda\r|^{p-2}\nabla u_\lambda + \mu(x)\l|\nabla u_\lambda\r|^{q-2}\nabla u_\lambda\r) \cdot \nabla  v \diff x\\
			&\quad +\eps \into \alpha(x) \l(u_\lambda\r)^{p-1} v \diff x+\eps \int_{\partial\Omega} \beta(x) \l(u_\lambda\r)^{p_* -1} v  \diff \sigma\\
			&\quad -\eps \into  \l( \zeta(x)\l(u_\lambda\r)^{-\kappa}+\lambda \l(u_\lambda\r)^{q_1-1}\r) v \diff x\\
			&\quad - \int_{\l\{ u_\lambda+\eps v< 0\r\}}\l(\l|\nabla u_\lambda\r|^{p-2}\nabla u_\lambda + \mu(x)\l|\nabla u_\lambda\r|^{q-2}\nabla u_\lambda\r) \cdot \nabla \l( u_\lambda+\eps v \r) \diff x\\
			&\quad -\int_{\l\{ u_\lambda+\eps v< 0\r\}} \alpha(x) \l(u_\lambda\r)^{p-1} \l (u_\lambda+\eps v\r) \diff x-\int_{\l\{ u_\lambda+\eps v< 0\r\}} \beta(x) \l(u_\lambda\r)^{p_* -1} \l (u_\lambda+\eps v\r) \diff \sigma\\
			&\quad +\int_{\{ u_\lambda+\e v< 0\}} \l(\zeta(x)\l(u_\lambda\r)^{-\kappa}+\lambda \l(u_\lambda\r)^{q_1-1}\r)\l( u_\lambda+\eps v \r)\diff x\\
			&\leq \eps \into \l ( \l|\nabla u_\lambda\r|^{p-2}\nabla u_\lambda + \mu(x)\l|\nabla u_\lambda\r|^{q-2}\nabla u_\lambda\r) \cdot \nabla  v \diff x\\
			&\quad +\eps \into \alpha(x) \l(u_\lambda\r)^{p-1} v \diff x+\eps \int_{\partial\Omega} \beta(x) \l(u_\lambda\r)^{p_* -1} v  \diff \sigma\\
			&\quad -\eps \into  \l( \zeta(x)\l(u_\lambda\r)^{-\kappa}+\lambda \l(u_\lambda\r)^{q_1-1}\r) v \diff x\\
			&\quad -\eps \int_{\l\{ u_\lambda+\eps v< 0\r\}}\l(\l|\nabla u_\lambda\r|^{p-2}\nabla u_\lambda + \mu(x)\l|\nabla u_\lambda\r|^{q-2}\nabla u_\lambda\r) \cdot \nabla v \diff x \\
			&\quad -\eps \int_{\l\{ u_\lambda+\eps v< 0\r\}} \alpha(x) \l(u_\lambda\r)^{p-1} v \diff x-\eps\int_{\l\{ u_\lambda+\eps v< 0\r\}} \beta(x) \l(u_\lambda\r)^{p_* -1} v \diff \sigma
	\end{align*}
	Note that the measure of the domain $\{u_\lambda+\eps h< 0\}$ tends to zero as $\eps \to 0$. Hence
	\begin{align*}
		& \int_{\l\{ u_\lambda+\eps v< 0\r\}}\l(\l|\nabla u_\lambda\r|^{p-2}\nabla u_\lambda + \mu(x)\l|\nabla u_\lambda\r|^{q-2}\nabla u_\lambda\r) \cdot \nabla v \diff x \to 0\quad\text{as }\eps \to 0,\\
		& \int_{\l\{ u_\lambda+\eps v< 0\r\}} \alpha(x) \l(u_\lambda\r)^{p-1} v \diff x\to 0\quad\text{as }\eps \to 0,\\
		& \int_{\l\{ u_\lambda+\eps v< 0\r\}} \beta(x) \l(u_\lambda\r)^{p_* -1} v \diff \sigma\to 0\quad\text{as }\eps \to 0.
	\end{align*}
	So, we can divide the inequality above with $\eps>0$ and let $\eps \to 0$. We obtain
	\begin{align*}
		\begin{split}
			& \into \Big(|\nabla u_\lambda|^{p-2} \nabla u_\lambda+ \mu(x) |\nabla u_\lambda|^{q-2} \nabla u_\lambda\Big) \cdot \nabla v \diff x\\
			& \quad +\into \alpha(x) (u_\lambda)^{p-1}v\diff x+\int_{\partial\Omega}\beta(x)(u_\lambda)^{p_* -1} v\diff \sigma \\
			&\geq \into \zeta(x)(u_\lambda)^{-\kappa} v \diff x+\lambda \into (u_\lambda )^{\nu-1}v\diff x.
		\end{split}
	\end{align*}
	Since $v\in\WH$ is arbitrary, equality must hold. It follows that $u_\lambda$ is a weak solution of problem \eqref{problem} such that $\Theta_\lambda(u_\lambda)<0$, 
	see Propositions \ref{proposition_negative_energy} and \ref{proposition_nonemptiness_and_existence}.
\end{proof}

Now we are interested in a second weak solution which turns out to be the global minimizer of $\Theta_\lambda$ restricted to $\mathcal{N}_\lambda^-$. First we show that this minimum is nonnegative.

\begin{proposition}\label{prop_minimizer_negative_manifold}
	If hypotheses \textnormal{(H)}  hold, then there exists $\lambda^* \in(0, \tilde{\lambda}]$ such that $\Theta_\lambda\big|_{\mathcal{N}^-_\lambda} > 0$ for all $\lambda \in(0, \lambda^*]$.
\end{proposition}

\begin{proof}
	First we take $u \in \mathcal{N}_\lambda^-$ which is possible by Proposition \ref{proposition_nonemptiness_and_existence}. Applying the definition of $\mathcal{N}_\lambda^-$ and the embedding $\Wp{p}\hookrightarrow \Lp{q_1}$ we have
	\begin{align*}
		\begin{split}
			 \lambda (q_1+\kappa-1)c_9^{q_1} \|u\|_{1,p}^{q_1} &\geq \lambda (q_1+\kappa-1) \|u\|_{q_1}^{q_1}\\
			&>(p+\kappa-1)\|u\|_{1,p}^p+(q+\kappa-1)\|\nabla u\|_{q,\mu}^q+(p_* +\kappa-1)\| u\|_{p_* ,\beta,\partial\Omega}^{p_* }\\
			&\geq (p+\kappa-1) \|u\|_{1,p}^p
		\end{split}
	\end{align*}
	for some constant $c_9>0$. This implies
	\begin{equation}\label{prop_23}
		\|u\|_{1,p}\geq
		\displaystyle \l[\frac{p+\kappa-1}{\lambda c_9^{q_1}(q_1+\kappa-1)}\r]^{\frac{1}{q_1-p}}.
	\end{equation}

	We argue by contradiction and suppose that the statement of the proposition is not true. Then there exists $u \in \mathcal{N}_\lambda^-$ such that $\Theta_\lambda(u)\leq 0$, that is,
	\begin{equation}\label{prop_24}
		\frac{1}{p}\|u\|_{1,p}^p +\frac{1}{q} \|\nabla u\|_{q,\mu}^q+\frac{1}{p_* } \| u\|_{p_* ,\beta,\partial\Omega}^{p_* }-\frac{1}{1-\kappa} \into\zeta(x)|u|^{1-\kappa}\diff x-\frac{\lambda}{q_1}\|u\|_{q_1}^{q_1}\leq 0.
	\end{equation}
	As $\mathcal{N}_\lambda^-\subseteq \mathcal{N}_\lambda$ we obtain from the definition of $\mathcal{N}_\lambda$
	\begin{equation}\label{prop_25}
		\frac{1}{q_1}\|\nabla u\|_{q,\mu}^q =\frac{1}{q_1}\into\zeta(x)|u|^{1-\kappa}\diff x+\frac{\lambda}{q_1} \|u\|_{q_1}^{q_1} -\frac{1}{q_1}\|u\|_{1,p}^p-\frac{1}{q_1}\|u\|_{p_* ,\beta,\partial\Omega}^{p_* }.
	\end{equation}
	Combining \eqref{prop_25} and \eqref{prop_24} one has
	\begin{align*}
		&\l(\frac{1}{p}-\frac{1}{q_1}\r)\| u\|_{1,p}^p+\l(\frac{1}{q}-\frac{1}{q_1}\r)\| \nabla u\|_{q,\mu}^q+\l(\frac{1}{p_* }-\frac{1}{q_1}\r)\|u\|_{p_* ,\beta,\partial\Omega}^{p_* }\\
		&+
		\l(\frac{1}{q_1}-\frac{1}{1-\kappa}\r)\into\zeta(x)|u|^{1-\kappa}\diff x\leq 0.
	\end{align*}
	Since $p_* <q_1$ and $q<q_1$ this implies
	\begin{align*}
		\frac{q_1-p}{pq_1}\|u\|_{1,p}^{p}\leq  \frac{q_1+\kappa -1}{q_1(1-\kappa)}\into\zeta(x) |u|^{1-\kappa}\diff x \leq \frac{q_1+\kappa -1}{q_1(1-\kappa)}c_{10}\|u\|_{1,p}^{1-\kappa}
	\end{align*}
	for some constant $c_{10}>0$. Therefore,
	\begin{equation}\label{prop_26}
		\|u\|_{1,p} \leq c_{11}
	\end{equation}
	for some $c_{11}>0$. Now we use \eqref{prop_26}  in \eqref{prop_23} and get
	\begin{align*}
		0<\frac{c_{12}}{c_{11}} \leq \lambda^{\frac{1}{q_1-p}}\quad\text{with}
		\quad \displaystyle c_{12}=\l[\frac{(p+\kappa-1)}{c_9^{q_1}(q_1+\kappa-1)}\r]^{\frac{1}{q_1-p}}>0.
	\end{align*}
	Letting $\lambda\to 0$ yields a contradiction since $1<p<q_1$. Hence, we find $\lambda^* \in(0, \tilde{\lambda}]$ such that $\Theta_\lambda\big|_{\mathcal{N}^-_\lambda} > 0$ for all $\lambda \in(0, \lambda^*]$.
\end{proof}

Next we will show that the functional $\Theta_\lambda$ achieves its global minimum restricted to the set $\mathcal{N}_\lambda^-$.

\begin{proposition}\label{proposition_minimize_negative}
	If hypotheses \textnormal{(H)}  hold and $\lambda \in(0, \lambda^*]$, then there exists $v_\lambda\in\mathcal{N}_\lambda^-$ with $v_\lambda\geq 0$ such that
	\begin{align*}
		m_\lambda^-=\inf_{\mathcal{N}_\lambda^-}\Theta_\lambda=\Theta_\lambda\l(v_\lambda\r)>0.
	\end{align*}
\end{proposition}

\begin{proof}
	First note that for $v\in \mathcal{N}_\lambda^-$ by using the embedding $\Wp{p} \hookrightarrow \Lp{q_1}$ we obtain
	\begin{align*}
		\begin{split}
		 &\lambda (q_1+\kappa-1) \|v\|_{q_1}^{q_1}
		 >(p+\kappa-1)\|v\|_{1,p}^p \geq (p+\kappa-1)\frac{1}{c_9^p}\|v\|_{q_1}^p
		\end{split}
	\end{align*}
	for $c_9>0$, see the proof of Proposition \ref{prop_minimizer_negative_manifold}. Therefore,
	\begin{equation}\label{prop_235}
		\|v\|_{q_1}\geq
		\displaystyle \l[\frac{p+\kappa-1}{\lambda c_9^{p}(q_1+\kappa-1)}\r]^{\frac{1}{q_1-p}}.
	\end{equation}
	
	Let $\{v_n\}_{n\in\N}\subset \mathcal{N}_\lambda^- \subset \mathcal{N}_\lambda$ be a minimizing sequence. Then, since $\mathcal{N}_\lambda^-\subset \mathcal{N}_\lambda$, we know from Proposition \ref{proposition_coerivity} that $\{v_n\}_{n\in\N}\subset \WH$ is bounded. We may assume that
	\begin{align*}
		v_n\weak v_\lambda \quad\text{in }\WH, \quad 
		v_n\to v_\lambda \quad\text{in }\Lp{q_1} \quad\text{and}\quad v_n\weak v_\lambda\quad\text{in }\Lprand{p_* }.
	\end{align*}
	From \eqref{prop_235} we see that $v_\lambda \neq 0$. 
	Now we will use the point $t^2_{v_\lambda}>0$ (see \eqref{prop_8}) for which we have
	\begin{align*}
		\eta_{v_\lambda}\l(t^2_{v_\lambda}\r)=\lambda \l\|v_\lambda\r\|_{q_1}^{q_1}\quad\text{and}\quad \eta'_{v_\lambda}\l(t^2_{v_\lambda}\r)<0.
	\end{align*}
	In the proof of Proposition \ref{proposition_nonemptiness_and_existence} we derived that $t^2_{v_\lambda} v_\lambda \in \mathcal{N}_\lambda^-$.
	
	Let us now show that $\lim_{n\to+\infty} \rho(v_n)=\rho(v_\lambda)$ for a subsequence (still denoted by $v_n$). As in the proof of Proposition \ref{proposition_nonemptiness_and_existence} we can prove Claims 1--3 via contradiction. Indeed, then we have in each case for a subsequence
	\begin{align*}
		\Theta_\lambda(t^2_{v_\lambda} v_\lambda)< \lim_{n\to\infty} \Theta_\lambda(t^2_{v_\lambda} v_n).
	\end{align*}
	Since $\Theta_\lambda(t^2_{v_\lambda}v_n) \leq \Theta_\lambda(v_n)$ (note that it is the global maximum since $\psi_{v_n}''(1)<0$, see Figure \ref{figure_1}) and $t^2_{v_\lambda} v_\lambda \in \mathcal{N}_\lambda^-$, we obtain
	\begin{align*}
		m_\lambda^- \leq \Theta_\lambda(t^2_{v_\lambda} v_\lambda) < m^-_\lambda,
	\end{align*}
	a contradiction. Therefore, for a subsequence, we have $\lim_{n\to+\infty} \rho(v_n)=\rho(v_\lambda)$ and since the integrand corresponding to the modular function $\rho(\cdot)$ is uniformly convex, this implies $\rho (\frac{v_n-v_\lambda}{2})\to 0$. Then Proposition \ref{proposition_modular_properties}\textnormal{(v)} implies that $v_n \to v_\lambda$ in $\WH$ and the continuity of $\Theta_\lambda$ gives  $\Theta_{\lambda}(v_n)\to \Theta_{\lambda}(v_\lambda)$ and so $\Theta_{\lambda}(v_\lambda)=m^-_{\lambda}$. 
	
	Since $v_n\in \mathcal{N}^-_{\lambda}$ for all $n\in \N$, we have
	\begin{align}\label{prop_15j2}
		\begin{split}
			&(p+\kappa-1)\|v_n\|_{1,p}^p+(q+\kappa-1)\|\nabla v_n\|_{q,\mu}^q+(p_* +\kappa-1)\| v_n\|_{p_* ,\beta,\partial\Omega}^{p_* }\\
			& -\lambda (q_1+\kappa-1) \|v_n\|_{q_1}^{q_1}<0.
		\end{split}	
	\end{align}
	Now we pass to the limit in \eqref{prop_15j2} as $n\to+\infty$ in order to get
	\begin{equation}\label{prop_162}
		(p+\kappa-1)\|v_\lambda\|_{1,p}^p+(q+\kappa-1)\|\nabla v_\lambda\|_{q,\mu}^q+(p_* +\kappa-1)\| v_\lambda\|_{p_* ,\beta,\partial\Omega}^{p_* }-\lambda (q_1+\kappa-1) \|v_\lambda\|_{q_1}^{q_1}\leq 0.
	\end{equation}
	From Proposition \ref{proposition_emptiness} we know that equality in \eqref{prop_162} cannot happen, so we have a strict inequality. Therefore, $v_\lambda\in \mathcal{N}^-_{\lambda}$. Since the treatment also works for $|v_\lambda|$ instead of $v_\lambda$, we may assume that $v_\lambda(x)\geq 0$ for a.\,a.\,$x\in\Omega$ such that $v_\lambda\neq 0$. Proposition \ref{prop_minimizer_negative_manifold} finally shows that $m_\lambda^->0$.
\end{proof}

Finally, we reach a second weak solution of problem \eqref{problem}.

\begin{proposition}\label{proposition_second_weak_solution}
	If hypotheses \textnormal{(H)}  hold and  $\lambda \in(0, \lambda^*)$, then $v_\lambda$ is a weak solution of problem \eqref{problem} such that $\Theta_\lambda(v_\lambda)>0$.
\end{proposition}

\begin{proof}
	Following the proof of Proposition \ref{proposition_energy_estimate} replacing $u_\lambda$ by $v_\lambda$ in the definition of $\Xi_h$ we can show for every $t \in [0,\delta_0]$ there exists $\chi(t)>0$ such that
	\begin{align*}
		\chi(t)\l(v_\lambda+th\r)\in \mathcal{N}_\lambda^-
		\quad\text{and}\quad
		\chi(t) \to 1 \quad\text{as } t\to 0^+.
	\end{align*}
	From Proposition \ref{proposition_minimize_negative} we derive that
	\begin{align}\label{prop_200}
		m_\lambda^-=\Theta_\lambda \l(v_\lambda\r) \leq \Theta_\lambda \l(\chi(t)\l(v_\lambda+th\r)\r)\quad\text{for all } t \in [0,\delta_0].
	\end{align}
	
	{\bf Claim:} $v_\lambda> 0$ for a.\,a.\,$x\in\Omega$
	
	Let us suppose there is a set $B$ with positive measure such that $v_\lambda=0$ in $B$. Let $h\in \WH$ with $h > 0$ and let $t\in (0,\delta_0)$, see \eqref{prop_200}, then $(\chi(t)(v_\lambda+th))^{1-\kappa}>(\chi(t)v_\lambda)^{1-\kappa}$ a.\,e.\,in $\Omega\setminus B$. From \eqref{prop_200} and since $\psi_{v_\lambda}(1)$ is the global maximum (see Figure \ref{figure_1}) which implies $\psi_{v_\lambda}(1) \geq \psi_{v_\lambda}(\chi(t))$ we then obtain 
	\begin{align*}
		0
		&\leq \frac{\Theta_\lambda(\chi(t)(v_\lambda+th))-\Theta_\lambda(v_\lambda)}{t} \\
		&\leq \frac{\Theta_\lambda(\chi(t)(v_\lambda+th))-\Theta_\lambda(\chi(t)v_\lambda)}{t}\\
		&=\frac{1}{p} \frac{\|\chi(t)(v_\lambda+th)\|_{1,p}^p-\|\chi(t)v_\lambda\|_{1,p}^p}{t} +\frac{1}{q} \frac{\|\nabla (\chi(t)(v_\lambda+th))\|_{q,\mu}^q-\|\nabla (\chi(t)v_\lambda)\|_{q,\mu}^q}{t}\\
		& \quad +\frac{1}{p_* } \frac{\|\chi(t)(v_\lambda+th)\|_{p_* ,\beta,\partial\Omega}^{p_* }-\|\chi(t)v_\lambda\|_{p_* ,\beta,\partial\Omega}^{p_* }}{t}-\frac{\chi(t)^{1-\kappa}}{(1-\kappa)t^{\kappa}} \int_B \zeta(x)h^{1-\kappa} \diff x\\
		&\quad - \frac{1}{1-\kappa}\int_{\Omega\setminus B} \zeta(x)\frac{(\chi(t)(v_\lambda+th))^{1-\kappa}-(\chi(t)v_\lambda)^{1-\kappa}}{t}\diff x\
		-\frac{\lambda}{q_1}
		\frac{\|\chi(t)(v_\lambda+th)\|_{q_1}^{q_1}-\|\chi(t)v_\lambda\|_{q_1}^{q_1}}{t}\\
		&<\frac{1}{p} \frac{\|\chi(t)(v_\lambda+th)\|_{1,p}^p-\|\chi(t)v_\lambda\|_{1,p}^p}{t} +\frac{1}{q} \frac{\|\nabla (\chi(t)(v_\lambda+th))\|_{q,\mu}^q-\|\nabla (\chi(t)v_\lambda)\|_{q,\mu}^q}{t}\\
		&\quad +\frac{1}{p_* } \frac{\|\chi(t)(v_\lambda+th)\|_{p_* ,\beta,\partial\Omega}^{p_* }-\|\chi(t)v_\lambda\|_{p_* ,\beta,\partial\Omega}^{p_* }}{t}
		-\frac{\chi(t)^{1-\kappa}}{(1-\kappa)t^{\kappa}} \int_B \zeta(x)h^{1-\kappa} \diff x\\
		&\quad -\frac{\lambda}{q_1}
		\frac{\|\chi(t)(v_\lambda+th)\|_{q_1}^{q_1}-\|\chi(t)v_\lambda\|_{q_1}^{q_1}}{t}.
	\end{align*}
	From the considerations above we see that
	\begin{align*}
		0
		&\leq \frac{\Theta_\lambda(\chi(t)(v_\lambda+th))-\Theta_\lambda(\chi(t)v_\lambda)}{t} \to -\infty \quad \text{as }t\to 0^+,
	\end{align*}
	which is a contradiction. This shows the Claim and so $v_\lambda>0$ a.\,e.\,in $\Omega$.
	
	The rest of the proof  works in the same way as the proof of Proposition \ref{proposition_first_weak_solution}. Indeed, \eqref{def1} and \eqref{def2} can be shown similarly using again  \eqref{prop_200} and the inequality $\psi_{v_\lambda}(1) \geq \psi_{v_\lambda}(\chi(t))$ along with $v_\lambda>0$. The last part of Proposition \ref{proposition_first_weak_solution} is the same replacing $u_\lambda$ by $v_\lambda$ and finally we know from Proposition \ref{proposition_minimize_negative} that $\Theta_\lambda(v_\lambda)>0$. 
\end{proof}

The proof of Theorem \ref{main_result} follows now from Propositions \ref{proposition_first_weak_solution} and \ref{proposition_second_weak_solution}.

\section*{Acknowledgment}

The first author was funded by the Deutsche Forschungsgemeinschaft (DFG, German Research Foundation) under Germany's Excellence Strategy – The Berlin Mathematics Research Center MATH+ and the Berlin Mathematical School (BMS) (EXC-2046/1, project ID: 390685689).


\end{document}